\documentclass[lettersize,journal]{IEEEtran}
\usepackage[caption=false,font=normalsize,labelfont=sf,textfont=sf]{subfig}

\usepackage{booktabs}
\usepackage{algorithmic}
\usepackage{algorithm}
\usepackage{amsfonts}
\usepackage{amsmath}
\usepackage{amssymb}
\usepackage{amsthm}
\usepackage{cases}
\usepackage{array}
\usepackage{arydshln}
\usepackage{bm}
\usepackage{cite}
\usepackage{color,xcolor,soul}

\usepackage{comment}
\usepackage{dsfont}
\usepackage{enumitem}
\usepackage{float}
\usepackage{graphicx}
\usepackage{grffile}
\usepackage{hyperref}
\hypersetup{colorlinks,linkcolor=blue,anchorcolor=blue,citecolor=blue}
\usepackage[utf8]{inputenc}
\usepackage[T1]{fontenc}
\usepackage{letltxmacro}
\usepackage{mathrsfs}
\usepackage{mathtools}
\usepackage{multirow}
\usepackage{nicefrac}
\usepackage{subcaption}
\usepackage{tablefootnote}
\usepackage{verbatim}
\usepackage{setspace}
\usepackage{makecell}

\usepackage{todonotes}

\definecolor{myPurple}{RGB}{128,0,128}

\newcommand{\br}{\bm{r}}

\newcommand{\bx}{\bm{x}}

\newcommand{\bB}{\bm{B}}

\newcommand{\bI}{\bm{I}}

\newcommand{\bP}{\bm{P}}

\newcommand{\bU}{\bm{U}}
\newcommand{\bV}{\bm{V}}
\newcommand{\bW}{\bm{W}}
\newcommand{\bX}{\bm{X}}

\newcommand{\bOmega}{\bm{\Omega}}

\newcommand{\cA}{\mathcal{A}}
\newcommand{\cB}{\mathcal{B}}

\newcommand{\cG}{\mathcal{G}}

\newcommand{\cL}{\mathcal{L}}

\newcommand{\cN}{\mathcal{N}}

\newcommand{\cS}{\mathcal{S}}

\newcommand{\cW}{\mathcal{W}}
\newcommand{\cX}{\mathcal{X}}
\newcommand{\cY}{\mathcal{Y}}
\newcommand{\cZ}{\mathcal{Z}}

\newcommand{\Matricize}[2]{\mathcal{M}_{#1}\left(#2\right)}

\newcommand{\norm}[1]{\left\lVert#1\right\rVert}
\newcommand{\inner}[2]{\left\langle#1,#2\right\rangle}

\newcommand{\argmin}{\mathop{\mathrm{argmin}}}

\allowdisplaybreaks
\makeatletter

\theoremstyle{plain}
\newtheorem{lemma}{\textbf{Lemma}} 

\newtheorem{theorem}{\textbf{Theorem}}

\theoremstyle{definition}\newtheorem{definition}{\textbf{Definition}}
 
\theoremstyle{remark}

\newtheorem{property}{\textbf{Property}}

\begin{document}

\title{Outlier-aware Tensor Robust Principal Component Analysis with Self-guided Data Augmentation}

\author{
  Yangyang Xu,
  Kexin Li,
  Li Yang,
  and You-Wei Wen
  \thanks{This work is supported by the National Natural Science Foundation of China (Grant No. 12361089); the Scientific Research Fund Project of Yunnan Provincial Education Department (Grant No. 2024J0642); the Yunnan Fundamental Research Projects (Grant Nos. 202401AU070104, 202401AU070105); the Scientific Research Fund Project of Yunnan University of Finance and Economics (Grant No. 2024D38).}
  \thanks{Yangyang Xu, Li Yang, and You-Wei Wen (corresponding author) are with the School of Mathematics and Statistics, Hunan Normal University, Changsha 410081, Hunan, China (email: \href{mailto:yangyangxu2002@gmail.com}{yangyangxu2002@gmail.com}; \href{mailto:liyang161029@gmail.com}{liyang161029@gmail.com}; \href{mailto:wenyouwei@gmail.com}{wenyouwei@gmail.com}).}
  \thanks{Kexin Li is with the School of Statistics and Mathematics, Yunnan University of Finance and Economics, Kunming 650221, Yunnan, China (email: \href{mailto:likx1213@163.com}{likx1213@163.com}).}
}

\markboth{Journal of \LaTeX\ Class Files,~Vol.~14, No.~8, August~2021}%
{Shell \MakeLowercase{\textit{et al.}}: A Sample Article Using IEEEtran.cls for IEEE Journals}


\maketitle

\begin{abstract}

Tensor Robust Principal Component Analysis (TRPCA) is a fundamental technique for decomposing multi-dimensional data into a low-rank tensor and an outlier tensor,  yet existing methods relying on sparse outlier assumptions often fail under structured corruptions. 
In this paper, we propose a self-guided data augmentation approach that employs adaptive weighting to suppress outlier influence, reformulating the original TRPCA problem into a standard Tensor Principal Component Analysis (TPCA) problem.
The proposed model involves an optimization-driven weighting scheme that dynamically identifies and downweights outlier contributions during tensor augmentation. 
We develop an efficient proximal block coordinate descent algorithm with closed-form updates to solve the resulting optimization problem, ensuring computational efficiency. 
Theoretical convergence is guaranteed through a framework combining block coordinate descent with majorization-minimization principles.
Numerical experiments on synthetic and real-world datasets, including face recovery, background subtraction, and hyperspectral denoising, demonstrate that our method effectively handles various corruption patterns. 
The results show the improvements in both accuracy and computational efficiency compared to state-of-the-art methods.
\end{abstract}

\begin{IEEEkeywords}
Tensor robust PCA, tensor PCA, data augmentation, outlier-aware strategy.
\end{IEEEkeywords}

\IEEEpeerreviewmaketitle

\section{Introduction}\label{section:introduction}

\IEEEPARstart{T}{he} rapid advancement of information technology has necessitated efficient processing of high-dimensional data across diverse applications, including video analysis, medical imaging, and remote sensing \cite{cai2024robust,jiang2020multi}. While Principal Component Analysis (PCA) \cite{wold1987principal} remains a cornerstone for dimensionality reduction in vectorized data, its inability to capture intricate structural relationships inherent in multi-way datasets, such as images and multi-modal measurements, has become increasingly apparent. 
This limitation has spurred the development of Tensor Principal Component Analysis (TPCA)\cite{tucker1966some,kroonenberg1980principal,kolda2009tensor,cichocki2015tensor}. 
% , a framework that inherently preserves multidimensional data characteristics through tensor representations \cite{kolda2009tensor,cichocki2015tensor}.  

The aim of TPCA lies in extracting a low-rank approximation tensor $\cL \in \mathbb{R}^{n_1 \times n_2 \times n_3}$ from an observed tensor $\cY \in \mathbb{R}^{n_1 \times n_2 \times n_3}$, thereby recovering latent patterns and structures. 
This can be formulated as an optimization problem \cite{liu2021tensors}:
\[
\min_{\operatorname{rank}(\cL)\leq k} \norm{\cL-\cY}^2_F, 
\]
where $\norm{\cdot}_F$ denotes the Frobenius norm, $\operatorname{rank}(\cdot)$ refers to the tensor rank, which varies depending on the chosen tensor decomposition method, and $k$ is the desired rank.
% The definition of tensor rank varies by decomposition approach, for example, CP rank \cite{carroll1970analysis,kiers2000towards}, Tucker rank \cite{tucker1966some,de2000introduction},  tubal rank \cite{kilmer2011factorization,kilmer2013third}, { tensor train rank \cite{oseledets2011tensor}, tensor ring rank \cite{zhao2016tensor} } and so on.

The rank minimization problem is well-known as NP-hard and usually replaced by the tensor nuclear norm minimization problem. 
Commonly employed tensor norms include: the Tensor Nuclear Norm (TNN) \cite{lu2019tensor}, Partial Sum of Tubal Nuclear Norm (PSTNN) \cite{jiang2020multi}, Weighted Tensor Schatten $p$-norm \cite{yang2022nonconvex}, and Weighted Tensor Nuclear Norm (WTNN) \cite{wang2023double}, and so on. 
While these approaches promote low-rankness, they often impose high computational costs due to their reliance on singular value decomposition (SVD).

Factorization-based approaches can avoid the issues associated with these tensor norms and offer both interpretability and physical meaning. 
These methods exploit the explicit structure of tensor decompositions, which allow for a more efficient and interpretable representation of tensor data. 
Representative factorization-based methods include the CANDECOMP/PARAFAC (CP) decomposition \cite{carroll1970analysis,kiers2000towards}, Tucker decomposition \cite{tucker1966some,de2000introduction}, Tensor Train (TT) decomposition \cite{oseledets2011tensor}, Tensor Ring (TR) decomposition \cite{zhao2016tensor}, and Fully-Connected Tensor Network (FCTN) decomposition \cite{zheng2021fully}. 
In this paper, we focus on the Tucker decomposition for tensor representation, where a  tensor $\cL$  is factorized  as $\cL = \left(\bU_1, \bU_2, \bU_3 \right) \cdot \cG$.  Here $\bU_{i} \in \mathbb{R}^{n_{i} \times r_{i}}$ are the factor matrices, $\cG \in \mathbb{R}^{r_{1} \times r_{2} \times r_{3}}$ is the core tensor, and $\br = [r_1, r_2, r_3]$ denotes the Tucker rank. 
Then the TPCA problem is reformulated as a low-rank factorization problem \cite{tucker1966some,kolda2009tensor}:
\begin{equation}\label{eq:TuckerDecompose}
    \min_{\cL = \left(\bU_1, \bU_2, \bU_3 \right) \cdot \cG} \left\| \cL - \cY \right\|_{F}^{2}.
\end{equation}
% where $\cL$ is the low-rank tensor, and $\cY$ is the target tensor to be approximated.

Traditional TPCA has demonstrated effectiveness in extracting low-dimensional structures from high-order data, with successful applications in image processing \cite{karami2012compression}, signal analysis \cite{cichocki2013tensor}, and data mining \cite{kolda2009tensor}. 
However, the presence of outliers and non-Gaussian noise in real-world measurements presents the challenges for TPCA methods. These corruptions in the observed tensor can distort the underlying low-rank structure, making it difficult to accurately recover the latent low-rank tensor.
Thus the Tensor Robust Principal Component Analysis (TRPCA) model has been developed \cite{lu2019tensor,gao2020enhanced}, 
where the observed tensor $\cX$ is the sum of  the low-rank tensor $\cL$ and the outlier tensor $\cS$ as follows
\[ \cX = \cL + \cS. \]
% where $\cL$ represents the low-rank tensor and $\cS$ captures outliers. 
Traditional approaches employ sparsity-inducing norms such as the $\ell_0$-norm or $\ell_1$-norm to characterize the outlier tensor $\cS$ \cite{huang2015provable,lu2016tensor}. 
Theoretical recovery guarantees for these methods rely on the assumption that outliers are sparse and randomly distributed, with the proportion of non-zero elements in $\cS$ remaining below a specified threshold \cite{huang2015provable, cai2023generalized}.
The TRPCA with $\ell_1$-norm regularization often underperforms as it uniformly processes tensor elements, disregarding the spatial-temporal coherence of structured outliers (e.g., contiguous regions in tensor slices for background subtraction), multi-regularization approaches for modeling $\cS$ have been developed \cite{wang2023double,alawode2025learning}. 
% In practical applications, TRPCA based on $\ell_1$-norm regularization often fails to deliver satisfactory results. The $\ell_1$-norm treats each tensor entry uniformly, failing to capture the spatial or temporal coherence of structured outliers, such as contiguous regions in tensor slices observed in background subtraction tasks. To address these shortcomings, the Double Auto-Weighted Tensor Robust Principal Component Analysis (DATRPCA) framework introduces a weighted $\ell_1$-norm approach, which adaptively assigns smaller weights to larger entries in $\cS$ to preserve significant, structured outliers while still enforcing sparsity \cite{wang2023double}. Furthermore, recent studies have proposed multi-regularization frameworks for modeling $\cS$, incorporating spatial-temporal constraints through graph-based priors or collaborative filtering \cite{alawode2025learning}. 

While these advanced methods improve handling of structured outliers, challenges remain when sparsity assumptions are violated. Moreover, when the ratio of non-zero elements in $\cS$ exceeds theoretical sparsity thresholds, the conventional sparse characterization becomes less appropriate, as tensors are typically considered sparse only when non-zero elements constitute a small fraction of the total elements. In such cases, characterizing $\cS$ as a sparse tensor may no longer be appropriate. The limitations of existing sparse outlier modeling strategies motivate the development of alternative approaches to recover $\cL$ without relying on assumptions about the sparsity of $\cS$. 
The outliers in the observed tensor distort the low-rank structure and  influence the factorization process. 

In this paper, we propose a self-guided data augmentation approach that employs adaptive weighting to suppress outlier influence, reformulating the original TRPCA problem into a standard TPCA problem.
% The proposed model involves an optimization-driven weighting scheme that dynamically identifies and downweights outlier contributions during tensor augmentation. 
% Conventional decomposition methods often increase errors when outliers distort the low-rank structure of $\cL$, as corrupted entries directly influence the factorization process. To address this, we propose a guided tensor augmentation framework that excludes outlier-contaminated entries while preserving the integrity of the underlying low-rank structure.  
% The methodology operates through weighted augmentation, iterative guidance and dynamic adaptation weight.
First, a weight tensor $\cW$ identifies outliers in the observed tensor $\cX$, with higher weights reflecting greater confidence in entries' alignment with the low-rank structure and lower weights flagging likely corruptions. Second, leveraging $\cW$, an augmented tensor $\cY$ is constructed as a weighted average of $\cX$ and the guidance tensor $\cL$. 
Through this data augmentation, the augmented tensor $\cY$ becomes a low-rank version of $\cX$, where outlier-corrupted entries are suppressed, and reliable entries are enhanced. 
This transformation effectively reformulates the TRPCA problem into a TPCA problem as defined in \eqref{eq:TuckerDecompose}, where the focus turns to recovering the low-rank structure from the augmented tensor. Third, the guidance tensor $\cL$ is refined at each iteration via low-rank factorization of $\cY$, progressively aligning with the underlying low-rank patterns. Finally, residual errors between $\cX$ and the augmented tensor $\cY$ drive updates to the weight tensor $\cW$, enabling automatic emphasis on statistically consistent entries and suppression of outliers. This self-guided process enhances $\cY$'s fidelity to the true low-rank structure across iterations, even when outliers exhibit dense or structured distributions. 
By decoupling outlier suppression from factorization, the framework eliminates dependence on explicit sparsity constraints while maintaining computational tractability through closed-form updates. 

The proposed framework introduces only two additions to standard low-rank tensor factorization: a weight tensor $\cW$ and an augmentation tensor $\cY$. Both components are computed via element-wise operations, avoiding significant computational overhead. Furthermore, each subproblem in the iterative process exhibits a quadratic structure, ensuring closed-form solutions at every step. Unlike tensor norm-based methods that rely on iterative Singular Value Decomposition (SVD) \cite{lu2019tensor,gao2020enhanced}, our factorization-based approach avoids SVD computations, making it particularly suited for large-scale tensor processing. Theoretical guarantees are established through a loss function minimized via block coordinate descent. We prove that the algorithm converges to critical points of the loss function.  

Overall, the contributions of this paper are as follows:
\begin{itemize}
\item We propose a novel outlier-aware TRPCA framework with self-guided data augmentation, leveraging a weight tensor for dynamic outlier identification and an augmented tensor for robust low-rank tensor recovery without sparsity assumptions. This framework transforms the TRPCA problem into standard TPCA by decoupling outlier suppression from low-rank approximation.
\item An efficient alternating minimization algorithm is developed with closed-form updates via proximal block coordinate descent. Leveraging the  block coordinate descent with majorization-minimization principles, the algorithm ensures global convergence under mild conditions with minimal computational overhead. 
\item Experimental results on synthetic and real-world datasets highlight the superior performance and efficiency of our method. On synthetic data, our approach is 20 times faster than the baseline method while achieving higher accuracy. In real-world applications, it performs face restoration in under one second, delivers rapid and top-performing background subtraction, and improves hyperspectral image denoising by 0.5-2 dB in PSNR with high efficiency.
\end{itemize}

The remainder of the paper is arranged as below. Section \ref{section:preliminary} introduces some notations and preliminaries. In Section \ref{section:objective}, the self-guided data augmentation process and the outlier-aware strategy are presented, leading to the proposed model, followed by the optimization algorithm. In Section \ref{section:convergence}, we provide a convergence analysis of the proposed algorithm. Section \ref{section:experiment} presents experimental results on both synthetic and real-world datasets. Finally, Section \ref{section:conclusion} concludes the paper.

\section{Notations and preliminaries}\label{section:preliminary}

We first introduce the notations in the paper. 
We use lowercase letters (e.g., $x$) to denote scalars, boldface lowercase letters (e.g., $\bx$) for vectors, boldface capital letters (e.g., $\bX$) for matrices, and Euler script letters (e.g., $\cX$) for tensors. The identity matrix is denoted by $\bI$, and $\mathbf{1}$ represents a tensor with all entries equal to one. 
The field of real tensors of size $n_1\times n_2\times n_3$ is represented as $\mathbb{R}^{n_1 \times n_2 \times n_3}$, and the $(i, j, k)$-th entry of a 3-order tensor $\cX\in \mathbb{R}^{n_1 \times n_2 \times n_3}$ is expressed as $\cX_{ijk}$ or $[\cX]_{ijk}$. The inner product of two tensors $\cA$ and $\cB$ is defined as $\langle \cA, \cB \rangle = \sum_{i,j,k} \cA_{ijk} \cB_{ijk}$, and the Frobenius norm of a tensor $\cX$ is given by $\|\cX\|_{F} = \sqrt{\langle \cX, \cX \rangle}$. The infinity norm of a tensor $\cX$ is denoted by $\|\cX\|_{\infty}=\max_{i,j,k}|\cX_{ijk}|$. Given a weight tensor $\cW\in \mathbb{R}^{n_1\times n_2\times n_3}$, the weighted Frobenius norm of tensor $\cX$ is defined as $\|\cX\|_{\cW}=\sqrt{\sum_{i,j,k}\cW_{ijk}\cX_{ijk}^{2}}$.

Additional notations include $\bX^{\mathrm{T}}$ for the transpose of a matrix $\bX$, $\bX^{-1}$ for the inverse of a matrix $\bX$, $\lceil x \rceil$ for the ceiling of a number $x$, $\otimes$ for the Kronecker product, and $\odot$ for the element-wise product of matrices or tensors. These notations, along with key definitions and properties introduced below, will be used throughout this paper. 

\begin{definition}[Mode-$k$ matricization \cite{cai2024robust}]
Given a $d$-mode tensor $\cX\in \mathbb{R}^{n_1 \times \cdots \times n_d}$, 
let $s$ denote the lexicographic index of the tuple $(i_1, \cdots, i_{k-1}, i_{k+1}, \cdots, i_d)$. 
The mode-$k$ matricization of $\cX$, 
denoted $\Matricize{k}{\cX} \in \mathbb{R}^{n_k \times \prod_{j \neq k} n_j}$, is defined such that  
\[
\left[\Matricize{k}{\mathcal{X}}\right]_{i_k, s} = \mathcal{X}_{i_1, \dots, i_{k-1}, i_k, i_{k+1}, \dots, i_d}.
\]  

This operation transforms the multi-dimensional tensor into a two-dimensional matrix by keeping only the $k$-th mode as the row indices of the resulting matrix.
\end{definition}

\begin{definition}[Mode-$k$ product \cite{cai2024robust}]
	The mode-$k$ product of a tensor $\cG\in \mathbb{R}^{n_1 \times \cdots \times n_d}$ and a matrix $\bU\in \mathbb{R}^{J \times n_k}$ is denoted by $\cX=\cG \times_{k} \bU$, which is defined as
	\begin{equation*}
		\cX_{i_1,\cdots,i_{k-1},j,i_{k+1},\cdots,i_d}=\sum_{s=1}^{n_{k}}\cG_{i_1,\cdots,i_{k-1},s,i_{k+1},\cdots,i_d}\bU_{js}.
	\end{equation*}
\end{definition}

\begin{definition}[Tucker decomposition and Tucker rank \cite{kolda2009tensor}]\label{def:Tucker decomposition}
	Given a tensor $\cX \in \mathbb{R}^{n_1 \times n_2 \times n_3}$, its Tucker decomposition approximates $\cX$ as a core tensor multiplied by factor matrices along each mode:
	\begin{equation*}
		\cX =\cG \times_1 \bU_1 \times_2 \bU_2 \times_3 \bU_3,
	\end{equation*}
	where $\bU_i \in \mathbb{R}^{n_i \times r_i}$ are the factor matrices with $r_i \leq n_i$, and $\cG \in \mathbb{R}^{r_1 \times r_2 \times r_3}$ is the core tensor. The tuple $\br = [r_1, r_2, r_3]$ is the Tucker rank, specifying the dimensions of $\cG$. For brevity, we denote this as $\cX = \left( \bU_1, \bU_2, \bU_3 \right) \cdot \cG$.
\end{definition}

\begin{property}[Properties of Tucker decomposition \cite{dong2023fast}] \label{property:tucker_decomp}
Let $\cX$ be a tensor with Tucker decomposition $\cX = \left( \bU_1, \bU_2, \bU_3 \right) \cdot \cG$, where $\bU_i$ are factor matrices and $\cG$ is the core tensor. The following properties hold:
\begin{enumerate}[leftmargin=*]
\item The matricization of $\cX$ along mode-$i$ admits the low-rank decomposition
\begin{equation*}
\Matricize{i}{\cX} = \bU_{i} \widehat{\bU}_i^{\mathrm{T}}, \quad i=1,2,3,
\end{equation*}
where
\begin{align*}
\widehat{\bU}_{1} &\coloneqq \left( \bU_3 \otimes \bU_2 \right) \Matricize{1}{\cG}^{\mathrm{T}}, \\
\widehat{\bU}_{2} &\coloneqq \left( \bU_3 \otimes \bU_1 \right) \Matricize{2}{\cG}^{\mathrm{T}}, \\
\widehat{\bU}_{3} &\coloneqq \left( \bU_2 \otimes \bU_1 \right) \Matricize{3}{\cG}^{\mathrm{T}}.
\end{align*}
\item For any tensor $\cY  \in \mathbb{R}^{n_{1} \times n_{2} \times n_{3}}$, we have the inner product equality
\begin{equation*}
\left\langle \left( \bU_1, \bU_2, \bU_3 \right) \cdot \cG, \cY \right\rangle = \left\langle \cG, \left( \bU_1^{\mathrm{T}}, \bU_2^{\mathrm{T}}, \bU_3^{\mathrm{T}} \right) \cdot \cY \right\rangle.
\end{equation*}
\item For any $\bW_i\in \mathbb{R}^{r_i\times r_i}$, 
\begin{align*}
        &\left(\bU_1\bW_1,\bU_2\bW_2,\bU_3\bW_3\right)\cdot\cG \\
        &=\left(\bU_1,\bU_2,\bU_3\right) \cdot\left(\left(\bW_1,\bW_2,\bW_3\right)\cdot\cG\right).
    \end{align*}
\end{enumerate}
\end{property}

\section{Proposed Method}\label{section:objective}

% In this section, we introduce the Self-guided Data Augmented Outlier-aware (SDAO) TRPCA model. The framework builds on two key ideas: reformulating TRPCA into a TPCA problem via self-guided data augmentation when outlier weights are known, and incorporating a outlier-aware strategy to dynamically detect and reduce the influence of outliers.

In this section, we present the Self-guided Data Augmented Outlier-aware (SDAO) TRPCA model, which transforms TRPCA into a TPCA problem through self-guided data augmentation that replaces outlier values. 
The framework operates in two modes: (1) an oracle case with known outlier locations, and (2) a practical case employing an outlier-aware strategy to dynamically detect and mitigate outliers.

\subsection{Self-guided Data Augmentation}\label{section:augmentation}

Under oracle case with known outlier locations, we establish the self-guided data augmentation framework.
Let $\bOmega$ denote the complete index set of tensor elements, with $\cN \subset \bOmega$ representing the identified outlier locations. 
We define an indicator tensor $\cW$ that explicitly encodes spatial outlier information:
\begin{equation}\label{eq:W}
	\cW_{ijk} = 
	\begin{cases}
		0, & (i,j,k) \in \cN. \\
		1, & (i,j,k) \notin \cN.
	\end{cases}
\end{equation}

In the TRPCA framework, the non-zero entries of $\cS$ correspond to outliers that  corrupt the tensor rank structure and hinder the standard TPCA implementation. 
To address these challenges, we develop a self-guided data augmentation process that generates an augmented  tensor $\cY$, which combines the observed tensor $\cX$ with a guidance tensor $\widetilde{\cL}$. At identified outlier locations ($ (i,j,k) \in \cN$ and $\cW_{ijk}=0$), values are fully replaced with corresponding entries from $\widetilde{\cL}$. For inlier locations ($ (i,j,k) \notin \cN$ and $\cW_{ijk}=1$), a convex combination balances the original observation $\cX$ with the guidance tensor $\widetilde{\cL}$. 
The augmented tensor construction follows a weighted formulation:
\begin{equation*}
    \cY_{ijk} = 
    \begin{cases}
        \widetilde{\cL}_{ijk}, & \cW_{ijk}=0, \\
        \beta \cX_{ijk} + (1-\beta)\widetilde{\cL}_{ijk}, & \cW_{ijk}=1,
    \end{cases}
\end{equation*}
where $\beta\in(0,1)$ is a weight factor.
This operation can be reformulated as a minimization problem:
\[
\cY =\argmin_{\cY}\left(\lambda \|\cY-\cX\|_{\cW}^2 + \|\widetilde{\cL} - \cY\|_F^2\right).
\]
where $\lambda = \frac{\beta}{1-\beta}$ and $\|\cdot\|_{\cW}$ denotes the weighted Frobenius norm. 
The resulting augmented tensor $\cY$ is then processed through standard TPCA in \eqref{eq:TuckerDecompose} to recover the low-rank component $\cL$, thereby separating outlier removal from low-rank approximation. 
Then, the optimal solution of the TPCA problem replaces $\widetilde{\cL}$ in subsequent iterations. 
This process naturally extends to an iterative refinement scheme:
\begin{equation}\label{updateYL}
\begin{cases} 
\cY_{t+1} = \argmin_{\cY} \left(\lambda \|\cY-\cX\|_{\cW}^2 + \|{\cL}_t - \cY\|_F^2\right), \\
\cL_{t+1} = \argmin_{\cL=(\bU_1,\bU_2,\bU_3)\cdot\cG} \norm{\cL-\cY_{t+1}}. 
\end{cases}
\end{equation}
This creates a self-guiding approach where each iteration's low-rank estimate $\cL_t$ generates an improved augmented tensor $\cY_{t+1}$,
which in turn yields a more accurate low-rank tensor $\cL_{t+1}$ through TPCA.
The cyclical refinement progressively mitigates outlier effects while enhancing recovery of the underlying low-rank structure, with the guidance tensor  converging to the true low-rank tensor over successive iterations.
This approach can be expressed as a joint optimization problem minimizing
\begin{equation}\label{DAobjfunW}
\min_{\cY,\cL = \left(\bU_1, \bU_2, \bU_3 \right) \cdot \cG} \Phi(\cY,\cL; \cW) = \lambda\|\cY-\cX\|_{\cW}^2 +  \|\cL - \cY\|_F^2.
\end{equation}

\subsection{Generalized Tensor Weight for Outlier Detection}\label{section:generalized_weight}

The oracle setting presumes exact knowledge of outlier locations $\cN$, constituting an idealized scenario. 
However, this assumption fails in practical scenarios where outliers must be detected from corrupted observations. 
This introduces significant challenges due to  distortion of the underlying low-rank structure and the interdependence between outlier detection and low-rank approximation.

When outliers correspond to extreme values in $\cX$, such as salt-and-pepper noise in image processing, their locations can be identified by comparing tensor elements to the maximum and minimum values, $x_{\max}$ and $x_{\min}$, respectively. In such cases, $\cW$ is defined as:
\begin{equation}\label{eq:W_impulsive}
    \cW_{ijk}=
\begin{cases}
0, &\text{if } \cX_{ijk}= x_{\max} \text{ or } x_{\min}, \\
1, &\text{otherwise}.
\end{cases}
\end{equation}
While this approach is computationally efficient, it may misclassify genuine extreme values as outliers. Nevertheless, in many applications, the proportion of such extreme values is small.

In more general cases, accurately detecting outliers requires advanced data analysis and signal processing techniques \cite{kim1995adaptive,schulte2006fuzzy}. To enhance robustness, we relax the binary constraint on $\cW$, allowing its elements to take continuous values in the range $[0, 1]$. This relaxation enables a  representation of non-outlier likelihoods. Specifically, the weight tensor $\cW$ is computed based on the discrepancy between the augmented tensor $\cY$ and the observed tensor $\cX$. Larger discrepancies indicate a higher likelihood of outliers, resulting in smaller weights, while smaller discrepancies correspond to normal data points, yielding weights closer to 1.
The weight tensor $\cW$ is  defined element-wise via exponential discrepancy measure:
\[
\cW_{ijk} = \exp\bigl(-\tfrac{1}{2\gamma}(\cY_{ijk} - \cX_{ijk})^2\bigr),
\]
where $\gamma>0$.

An important theoretical property emerges as $\gamma \to 0$, where the continuous weights converge to the ideal binary indicator:
\begin{equation}
\lim_{\gamma\to 0}\cW_{ijk}=
\begin{cases}
0, & \cY_{ijk} \neq \cX_{ijk}, \\
1, & \cY_{ijk} = \cX_{ijk}.
\end{cases}
\end{equation}
This limiting behavior guarantees backward compatibility with the oracle setting.

\subsection{Numerical Algorithm}\label{section:Numerical Algorithm}

The proposed algorithm implements an iterative optimization framework that alternates between updating the weight tensor $\cW$, augmented tensor $\cY$, and low-rank tensor $\cL$. This iterative process creates a mutually dependent system where the weight tensor $\cW$  influences the subsequent updates of $\cY$, while itself being determined by previous iterations.
We remark that in practical applications, all observed data values are inherently bounded. This physical constraint necessarily requires that all iterates in our optimization procedure remain bounded within a prescribed range. 
This is that the augmented tensor $\cY$ should lie in the set:
$$\bB = \{ \cY : \|\cY\|_\infty \leq a \}$$
for some finite positive constant $a$, which is set to $\|\cX\|_\infty$ to match the range of the original observed data. 
% The iterative updates are driven by minimizing a joint objective function that integrates data augmentation for outlier suppression and low-rank approximation (TPCA) for the given outlier-aware weight $\cW_{t+1}=\omega(\cY_t)$:
% \begin{equation}
%     \min_{\cY,\cL = \left(\bU_1, \bU_2, \bU_3 \right) \cdot \cG} \lambda\|\cY-\cX\|_{\cW_{t+1}}^2 + \|\cL - \cY\|_F^2.
% \end{equation}
% where $\cW$, $\cY$, and $\cL$ represent the outlier-aware weight, augmented tensor, and low-rank component, respectively. Based on this formulation, 
Consider the update for the outlier-aware weight $\cW_t$, we modify the iterative scheme in \eqref{updateYL} as follows: 
% the complete algorithmic procedure is formally expressed as:
\begin{equation}\label{eq:algorithm}
\begin{cases}
\cW_{t+1} = \omega(\cY_t),\\
\cY_{t+1}  = \argmin_{\cY\in \bB}\Phi(\cY,\cL_{t}; \cW_{t+1}), \\
\cL_{t+1}  = \argmin_{\cL=(\bU_1,\bU_2,\bU_3)\cdot\cG}\Phi(\cY_{t+1},\cL; \cW_{t+1}). 
\end{cases}
\end{equation}
Here $\Phi$ is given in \eqref{DAobjfunW} and the function $\omega(\cZ)$ is defined by
\begin{equation}\label{UpdateW}
\omega(\cZ)=\exp\bigl(-\tfrac{1}{2\gamma} (\cZ-\cX)^2\bigr),
\end{equation}
where the squaring operation is element-wise.

\subsubsection{Subproblem for $\cY$} 

The optimization subproblem for $\cY_{t+1}$ possesses a separable structure, allowing decomposition into element-wise minimization problems:  
\[
\min_{|\cY_{ijk}|\leq a} \left[ \lambda [\cW_{t+1}]_{ijk}(\cY_{ijk} - \cX_{ijk})^2 + (\cY_{ijk} - [\cL_{t}]_{ijk})^2 \right],
\]  
where the constraint $|\cY_{ijk}| \leq a$ enforces boundedness. This separability yields the closed-form solution:  
\begin{equation}\label{UpdateY}
\cY_{t+1} =\bP(\bigl(\lambda\cW_{t+1} \odot \cX  + {\cL}_{t} \bigr)\oslash (\lambda\cW_{t+1}+\mathbf{1})). 
\end{equation}  
where $\bP(\cdot)$ denotes projection onto the bounded set $\bB$, truncating values outside $\bB$ to the nearest boundary, $\odot$ and $\oslash$ represent Hadamard product and element-wise division respectively, and $\mathbf{1}$ is an all-ones tensor. 

\subsubsection{Subproblem for $\cL$}

The subproblem for the low-rank tensor $\cL_{t+1}$ can be reformulated into a TPCA problem:
\[
\min_{\cL = \left(\bU_1,\bU_2,\bU_3\right)\cdot\cG} \|\cL - \cY_{t+1}\|_F^2,
\]
where $\left(\bU_1,\bU_2,\bU_3\right)\cdot\cG$ is the Tucker decomposition form of $\cL$. We apply the Block Coordinate Descent (BCD) method \cite{razaviyayn2013unified,xu2013block} to find $\cL_{t+1}$. 
Unlike traditional TPCA, which requires multiple iterations to converge, our approach treats the factors and core tensor as intermediate variables within an alternating optimization framework. Instead of fully solving the TPCA subproblem at each outer iteration, we perform a single PBCD update step per block $(\bU_{1}, \bU_{2}, \bU_{3}, \cG)$,  improving computational efficiency. 
% Consequently, rather than fully resolving the TPCA subproblem at each outer iteration, we execute a single PBCD update step per block $(\bU_{1},\bU_{2},\bU_{3},\cG)$, thereby enhancing computational efficiency. Specifically, 
The updates for $(\bU_{1,t+1},\bU_{2,t+1},\bU_{3,t+1},\cG_{t+1})$ are given by
\begin{align}
	\bU_{1,t+1}&=\argmin_{\bU_1}\norm{\left(\bU_1,\bU_{2t},\bU_{3t} \right)\cdot \cG_{t}-\cY_{t+1}}_{F}^{2}  \nonumber \\ 
	&\qquad \qquad +\alpha_{1} \left\| \bU_1-\bU_{1t}\right\|^2_{F},\label{Subproblem U1}\\
\bU_{2,t+1}&=\argmin_{\bU_2}\norm{\left(\bU_{1,t+1},\bU_{2},\bU_{3t} \right)\cdot \cG_{t}-\cY_{t+1}}_{F}^{2}    \nonumber \\ 
	&\qquad \qquad +\alpha_{1} \left\| \bU_2-\bU_{2t}\right\|^2_{F},\label{Subproblem U2}\\
\bU_{3,t+1}&=\argmin_{\bU_3}\norm{\left(\bU_{1,t+1},\bU_{2,t+1},\bU_{3} \right)\cdot \cG_{t}-\cY_{t+1}}_{F}^{2}    \nonumber \\ 
	&\qquad \qquad +\alpha_{1} \left\| \bU_3-\bU_{3t}\right\|^2_{F},\label{Subproblem U3}\\
\cG_{t+1}&=\argmin_{\cG}\norm{\left(\bU_{1,t+1},\bU_{2,t+1},\bU_{3,t+1} \right)\cdot \cG-\cY_{t+1}}_{F}^{2}    \nonumber \\ 
	& \qquad \qquad +\alpha_{2}^3
   \left\| \cG-\cG_{t}\right\|^2_{F}, \label{Subproblem G}
\end{align}
where the second term in each equation is the proximal term and $\alpha>0$ is the regularization parameter. 

% \subsubsection*{Subproblem for Tensor factors}
Now, we consider the solution for the factor matrices in \eqref{Subproblem U1}--\eqref{Subproblem U3}. 
Using matricization and Property \ref{property:tucker_decomp}, the subproblems for $\bU_{i}(i=1,2,3)$ can be rewritten as: 
\begin{equation*}
    \bU_{i,t+1}=\argmin_{\bU_i}\left\|\bU_i\widehat{\bU}_{it}^{\mathrm{T}}-\Matricize{i}{\cY_{t+1}}\right\|^2_{F}+\alpha_{1}\left\|\bU_{i}-\bU_{it}\right\|_{F}^2,
\end{equation*}
where $\widehat{\bU}_{1t} \coloneqq  \big(\bU_{3t} \otimes \bU_{2t} \big) \Matricize{1}{\cG_{t}}^{\mathrm{T}} $, $\widehat{\bU}_{2t}\coloneqq  \big(\bU_{3t} \otimes \bU_{1,t+1} \big) \Matricize{2}{\cG_{t}}^{\mathrm{T}} $ and $\widehat{\bU}_{3t}\coloneqq  \big(\bU_{2,t+1} \otimes \bU_{1,t+1} \big) \Matricize{3}{\cG_{t}}^{\mathrm{T}} $.
Consequently, we can obtain:
\begin{equation}\label{UpdateU}
    \bU_{i,t+1}=\bigl( \Matricize{i}{\cY_{t+1}}\widehat{\bU}_{it} + \alpha_{1} \bU_{it} \bigr)\bigl(\widehat{\bU}_{it}^{\mathrm{T}}\widehat{\bU}_{it}+\alpha_{1}\bI\bigr)^{-1}.
\end{equation}
We remark that  the factor matrices can be regarded as being updated by the scaled gradient descent \cite{dong2023fast}. 
% % This is because $\bU_{i,t+1}$ can also be rewritten in the following form:
% \begin{equation*}
%     \bU_{i,t+1}=\bU_{it} -  \nabla_{\bU_{i}}f_i(\bU_{it})\left(\widehat{\bU}_{it}^{\mathrm{T}}\widehat{\bU}_{it}+\alpha_{1}\bI\right)^{-1}.
% \end{equation*}
% Here the function $f_i$ is defined as
% \[
% f_i(\bU_{i})=\frac{1}{2}\left\|\bU_i\widehat{\bU}_{it}^{\mathrm{T}}-\Matricize{i}{\cY_{t+1}}\right\|^2_{F}.
% \]
% % This formulation provides an alternative perspective on the update process of the tensor factors, highlighting the relationship with the scaled gradient descent. 
% The gradient $\nabla_{\bU_{i}}f_i$ determines the direction and magnitude of the update, while the term $(\widehat{\bU}_{it}^{\mathrm{T}}\widehat{\bU}_{it}+\alpha_{1}\bI)^{-1}$  serves as a scaling factor that modifies the influence of the gradient. 

% \subsubsection*{Subproblem for Core tensor}
Next, we consider the solution for core tensor in \eqref{Subproblem G}. 
% The core tensor $\cG$ is updated by solving the problem \eqref{Subproblem G}. 
Using Property~\ref{property:tucker_decomp}, it can be equivalently rewritten as:
\[
\cG_{t+1}=\argmin_{\cG}\inner{(\bV_{1,t+1}, \bV_{2,t+1}, \bV_{3,t+1})\cdot\cG}{\cG}-\inner{\cB_{t+1}}{\cG}.
\]
Here $\bV_{i,t+1}=\bU_{i,t+1}^T\bU_{i,t+1}+\alpha_{2} \bI$ and  $\cB_{t+1}=\alpha_{2}^3\cG_t+(\bU_{1,t+1}^T, \bU_{2,t+1}^T, \bU_{3,t+1}^T)\cdot\cY_{t+1}$. 
Hence, the closed-form solution is given by
\begin{equation}\label{UpdateG}
\cG_{t+1}=(\bV_{1,t+1}^{-1}, \bV_{2,t+1}^{-1}, \bV_{3,t+1}^{-1})\cdot\cB_{t+1}.
\end{equation}
% After reconstructing the low-rank component $\cL_{t+1}$ from the core tensor and factor matrices, the sparse component can be obtained by $\cS_{t+1}=\cX-\cL_{t+1}$.

Now, we summarize the resulting algorithm for Self-guided Data Augmented Outlier-aware (SDAO) TRPCA in Algorithm \ref{alg:SDAO-tensor robust PCA}. 
We remark that in the oracle case and for impulsive noise, updating the weight $\cW_{t + 1}$ according to the algorithm is unnecessary. 
Instead, the weight $\cW_{t + 1}$ is substituted with the predefined weight specified in \eqref{eq:W} or \eqref{eq:W_impulsive}, respectively.

\begin{algorithm}[htbp]
	\renewcommand{\algorithmicrequire}{ \textbf{Input:}}
	\renewcommand{\algorithmicensure}{ \textbf{Output:}}
	\caption{Self-guided Data Augmented Outlier-aware (SDAO) TRPCA}
	\label{alg:SDAO-tensor robust PCA}
	\begin{algorithmic}[1]
		\REQUIRE The observed tensor $\cX\in \mathbb{R}^{n_{1}\times n_{2}\times n_{3}}$, Tucker rank $\br$, maximum
		iteration $T$, and parameter $\gamma$.
		\STATE \textbf{Initialize: }$\cY_0 = \operatorname{Tucker}(\cX, \br)=\left( \bU_{1,0}, \bU_{2,0}, \bU_{3,0} \right) \cdot \cG_{0}$, where $\cY_0$ is the rank-$\br$ Tucker decomposition of $\cX$; set $\alpha_{1}=\alpha_{2} = 10^{-10}$, $\lambda = 1$, and $\epsilon = 10^{-8}$.
		\FOR{$t=0,1,\cdots,T$}
        \STATE $\cW_{t+1} = \omega(\cY_t)$;% (see \eqref{UpdateW});
        \STATE $\cY_{t+1} =
        \bP\left(\lambda\cW_{t+1} \odot \cX  + {\cL}_{t} \right)\oslash (\lambda\cW_{t+1}+\mathbf{1})$;% (see \eqref{UpdateY});
		\STATE $\bU_{i,t+1}=\bigl( \Matricize{i}{\cY_{t+1}}\widehat{\bU}_{it} + \alpha_{1} \bU_{it} \bigr)\bigl(\widehat{\bU}_{it}^{\mathrm{T}}\widehat{\bU}_{it}+\alpha_{1}\bI\bigr)^{-1}$ for $i=1,2,3$;% (see \eqref{UpdateU}); %Update the factor matrices $\bU_{i,t+1},i = 1, 2, 3$ via \eqref{UpdateU};
		\STATE $\cG_{t+1}=(\bV_{1,t+1}^{-1}, \bV_{2,t+1}^{-1}, \bV_{3,t+1}^{-1})\cdot\cB_{t+1}$; %(see \eqref{UpdateG});%Update the core tensor $\cG_{t+1}$ via \eqref{UpdateG};
        \STATE $\cL_{t+1}=\left( \bU_{1,t+1}, \bU_{2,t+1}, \bU_{3,t+1} \right)\cdot\cG_{t+1}$;
		\STATE Stop if the convergence conditions are met:
		$$\begin{aligned}
		    &\left\|\cL_{t+1}-\cL_{t}\right\|_{\infty} \leq \epsilon,\quad\left\|\cY_{t+1}-\cY_{t}\right\|_{\infty} \leq \epsilon, \\
            & \left\|\cL_{t+1}-\cY_{t+1}\right\|_{\infty} \leq \epsilon.
		\end{aligned}
		$$
		\ENDFOR
		\ENSURE $\cL=\cL_{t+1}$ and  $\cS=\cX-\cL$.
	\end{algorithmic}
\end{algorithm}

\section{Convergence Analysis}\label{section:convergence}

The iterative scheme in \eqref{eq:algorithm} presents analytical challenges due to its alternating update structure. 
Unlike conventional optimization methods derived from a single explicit objective function, our algorithm combines three distinct yet interdependent update mechanisms. 
The weight tensor $\cW_{t+1}$ is computed through the nonlinear transformation $\omega(\cY_t)$ defined in \eqref{UpdateW}. This weight update then informs the subsequent computation of $\cY_{t+1}$ through a weighted optimization problem that incorporates both the updated weights and the previous low-rank estimate $\cL_t$. 
Meanwhile, the Tucker factors $\bU_i$ and core tensor $\cG$ follow decomposition procedures that remain independent of the weight tensor $\cW$.

To address the challenges posed by this coupling, we develop a convergence framework that combines block coordinate descent with majorization-minimization principles. 
We introduce a surrogate objective function $\Psi(\cY,\bU_1,\bU_2,\bU_3,\cG)$ with the properties: (1) the weight sequence $\{\cW_t\}$ is derived from the optimality conditions of $\Psi$'s majorizing function, and (2) the variable sequence $\{\cY_t,\bU_{1t},\bU_{2t},\bU_{3t},\cG_t\}$ is generated through block-wise minimization of $\Psi$. 

Since $\cL = (\bU_1, \bU_2, \bU_3)\cdot\cG$, we simplify $\|(\bU_1, \bU_2, \bU_3)\cdot\cG - \cY\|_{F}^2$ as $\|\cL - \cY\|_{F}^2$ and the sequence $\{\cY_t,\bU_{1t},\bU_{2t},\bU_{3t},\cG_t\}$  as the sequence $\{\cY_t,\cL_t\}$  when no ambiguity arises. 
Our analysis begins by introducing the Welsch's function \cite{holland1977robust}, a smooth non-convex objective function, which serves as our robust loss function:
\begin{equation}
\psi(\cY) = 2\gamma\sum_{ijk}\bigl(1-\exp\bigl(-\tfrac{(\cY_{ijk}-\cX_{ijk})^2}{2\gamma}\bigr)\bigr).
\end{equation}
We consider the surrogate function of $\psi(\cY)$. 
This function has a quadratic surrogate:
\begin{equation}
\widehat{\psi}(\cY; \cZ) = \psi(\cZ) + \|\cY-\cX\|_{\omega(\cZ)}^2-\|\cZ-\cX\|_{\omega(\cZ)}^2,
\end{equation}
where $\omega(\cZ_{ijk}) = \exp\bigl(-\tfrac{1}{2\gamma}(\cZ_{ijk}-\cX_{ijk})^2\bigr)$ defines the weights. 
A surrogate function is a locally tight upper bound of the original objective function that facilitates optimization by decomposing the problem into simpler subproblems. 
It is easy to check that  $\widehat{\psi}(\cY; \cZ)$ at a reference point $\cZ$ satisfies:  
\begin{enumerate}
\item upper bound property: $\psi(\cY) \leq \widehat{\psi}(\cY; \cZ)$ for all $\cY,\cZ$.
\item local tightness: $\psi(\cZ) = \widehat{\psi}(\cZ; \cZ)$.
\item first-order consistency: $\nabla \psi(\cZ) = \nabla \widehat{\psi}(\cZ; \cZ)$.
\end{enumerate}

The complete objective function and its surrogate are:
\begin{align}
\Psi(\cY,\cL) &= \|\cL - \cY\|_{F}^2 + \lambda \psi(\cY), \label{objfunPsi} \\
\widehat{\Psi}(\cY,\cL; \cZ) &= \|\cL - \cY\|_{F}^2 + \lambda \widehat{\psi}(\cY; \cZ).\label{objfunhatPsi}
\end{align}
One can readily verify that $\widehat{\Psi}(\cY,\cL; \cZ)$ serves as a surrogate function for $\Psi(\cY,\cL)$ with respect to the variable $\cY$.

\begin{lemma}
Let  $\Psi(\cY,\cL)$ and $\widehat\Psi(\cY,\cL; \cZ)$ be defined in  \eqref{objfunPsi} and  \eqref{objfunhatPsi} respectively, and $\{\cY_t\}$ be the sequences generated by Algorithm \ref{alg:SDAO-tensor robust PCA}.
Then we have
\begin{equation}
\cY_{t+1} = \argmin_{\cY\in \bB} \widehat{\Psi}(\cY,\cL_t;\cY_t),
\end{equation}
and the objective satisfies:
\begin{equation}\label{decreaseY}
\Psi(\cY_{t+1},\cL_{t})\leq % \widehat\Psi(\cY_{t+1},\cL_{t};\cY_{t})\leq \widehat\Psi(\cY_{t},\cL_{t};\cY_{t})=
\Psi(\cY_{t},\cL_{t}).
\end{equation}
\end{lemma}
 
\begin{proof}
The first equality follows directly from the algorithmic construction. To establish the descent property, we observe that by the minimization property of $\cY_{t+1}$:  
\[
\widehat{\Psi}(\cY_{t+1},\cL_t;\cY_t) \leq \widehat{\Psi}(\cY_t,\cL_t;\cY_t) = \Psi(\cY_t,\cL_t),
\]  
where the equality holds because the surrogate function $\widehat{\Psi}$ matches the original objective $\Psi$ at $\cY_t$. Furthermore, the surrogate function $\widehat{\Psi}$ majorizes the original objective $\Psi$ for any $\cY$, which gives:  
\[
\Psi(\cY_{t+1},\cL_t) \leq \widehat{\Psi}(\cY_{t+1},\cL_t;\cY_t).
\]  
Hence the lemma holds. 
\end{proof}

For the factor matrix updates and the core tensor update, we have the following lemma. 
\begin{lemma}
Let $\{\bU_{1t}, \bU_{2t}, \bU_{3t},\cG_{t}\}$ be  the sequences generated by \eqref{Subproblem U1}--\eqref{Subproblem G}, and $\cL_{t}=(\bU_{1t}, \bU_{2t}, \bU_{3t})\cdot\cG_{t}$,  then we have
\begin{align*}
    &\norm{\cL_{t+1}-\cY_{t+1}}_{F}^{2} +\alpha_1\sum_{i=1}^3\norm{\bU_{i,t+1}-\bU_{it}}_F^2\\
    &\quad+\alpha_2\|\cG_{t+1}-\cG_t\|_{F}^2 \leq \norm{\cL_{t}-\cY_{t+1}}_{F}^{2}.
\end{align*}
Moreover, we have 
\begin{align*}
    &\Psi(\cY_{t+1},\cL_{t+1})+\alpha_1\sum_{i=1}^3\norm{\bU_{i,t+1}-\bU_{it}}_F^2\\
    &\quad+\alpha_2\|\cG_{t+1}-\cG_t\|_{F}^2 
\leq \Psi(\cY_{t+1},\cL_{t}).
\end{align*}
\end{lemma}
\begin{proof}
The proof proceeds by analyzing each block update sequentially. 
From \eqref{Subproblem U1}, the optimality of $\bU_{1,t+1}$ yields:
\begin{align*}
\norm{(\bU_{1,t+1},\bU_{2t},\bU_{3t})\cdot\cG_{t}-\cY_{t+1}}_{F}^{2} + \alpha_{1}\norm{\bU_{1,t+1}-\bU_{1t}}^2_{F} & \\
\leq \norm{(\bU_{1t},\bU_{2t},\bU_{3t})\cdot\cG_{t}-\cY_{t+1}}_{F}^{2}.&
\end{align*}
Similarly, \eqref{Subproblem U2} gives:
\begin{align*}
\norm{(\bU_{1,t+1},\bU_{2,t+1},\bU_{3t})\cdot\cG_{t}-\cY_{t+1}}_{F}^{2} + \alpha_{1}\norm{\bU_{2,t+1}-\bU_{2t}}^2_{F} & \\
\leq \norm{(\bU_{1,t+1},\bU_{2t},\bU_{3t})\cdot\cG_{t}-\cY_{t+1}}_{F}^{2}.&
\end{align*}
From \eqref{Subproblem U3}, we have
\begin{align*}
&\norm{(\bU_{1,t+1},\bU_{2,t+1},\bU_{3,t+1})\cdot\cG_{t}-\cY_{t+1}}_{F}^{2} + \alpha_{1}\norm{\bU_{3,t+1}-\bU_{3t}}^2_{F} \\
&\qquad \qquad\leq \norm{(\bU_{1,t+1},\bU_{2,t+1},\bU_{3t})\cdot\cG_{t}-\cY_{t+1}}_{F}^{2}.
\end{align*}
Finally, \eqref{Subproblem G} provides:
\begin{align*}
&\norm{\cL_{t+1}-\cY_{t+1}}_{F}^{2} + \alpha_{2}\norm{\cG_{t+1}-\cG_{t}}^2_{F}  \\
\leq &\norm{(\bU_{1,t+1},\bU_{2,t+1},\bU_{3,t+1})\cdot\cG_{t}-\cY_{t+1}}_{F}^{2}.
\end{align*}
Summing these inequalities, the first result holds. 
The second inequality follows by  adding $\lambda\psi(\cY_{t+1})$ to both sides.
\end{proof}

The following lemma states that the sequence $\{(\cY_t,\cL_t)\}$ is non-increasing with respect to $\Psi(\cY,\cL)$. 
\begin{lemma}
Let $\{(\cY_t, \cL_{t})\}$ be the sequences generated by Algorithm \ref{alg:SDAO-tensor robust PCA}.
Then we have 
\begin{align*}
    &\Psi(\cY_{t+1},\cL_{t+1})+\alpha_1\sum_{i=1}^3\norm{\bU_{i,t+1}-\bU_{it}}_F^2\\
    &\quad+\alpha_2\|\cG_{t+1}-\cG_t\|_{F}^2 
\leq \Psi(\cY_{t},\cL_{t}).
\end{align*}
Moreover, we have $\lim_{t\rightarrow \infty}\|\bU_{i,t+1}-\bU_{it}\|_F^2=0$ for $i=1,2,3$ and $\lim_{t\rightarrow \infty} \norm{\cG_{t+1}-\cG_{t}}_F^2=0$. Consequently, $\lim_{t\rightarrow \infty} \norm{\cL_{t+1}-\cL_{t}}_F^2=0$. 
\end{lemma}

The iterative process maintains $\cY_t \in \bB$ for all iterations $t$, where $\bB = \{\cY : \|\cY\|_\infty \leq a\}$ represents the prescribed bounded domain. 
The boundedness of $\{\cY_t\}$ directly implies the boundedness of the corresponding low-rank estimates $\{\cL_t\}$.
Hence we have the following convergence result. 
\begin{theorem}
Assuming the set $\bOmega$ is convex, let $\bx_{t+1}:=(\cY_{t+1},\cL_{t+1})\in \bOmega$ denote the variables in the $(t+1)$-th iteration of Algorithm \ref{alg:SDAO-tensor robust PCA}. Then every limit point of $\{\bx_{t+1}\}$ is a stationary point of $\Psi(\cY,\cL)$, i.e.,
\begin{equation}
\lim _{t \rightarrow \infty} d\left(\bx_{t+1}, \bOmega^{*}\right)=\lim _{t \rightarrow \infty}\inf_{\bx^* \in \bOmega^{*}}\|\bx_{t+1}-\bx^*\|=0,
\end{equation}
where $\bOmega^{*}$ is the set of stationary points of $\Psi(\cY,\cL)$.
\end{theorem}

The detailed proof is provided in the supplementary material.

\section{Experimental Results}\label{section:experiment}

In this section, we apply the proposed method to both simulations and real-world experiments. We first evaluate tensor recovery performance under varying Tucker ranks and sparse noise levels, followed by a study of regularization parameters. Next, we test the algorithm on face denoising, hyperspectral image denoising, and background subtraction. In all experiments, we set $\alpha_1 =\alpha_2= 10^{-10}$, $\epsilon=10^{-8}$ and $\lambda = 1$. The parameter $\gamma$ is selected according to the formula:
$$\gamma=\frac{\gamma_0}{n_{1}n_{2}n_{3}}\sum_{i,j,k}\left[\cY_{0}-\cX\right]_{ijk}^2,$$ 
where $\gamma_0$ is a pre-defined parameter. Specifically, we set $\gamma_0 = 0.5$ for background subtraction and $\gamma_0 = 0.05$ for the other experiments. The Tucker rank values differ across experiments, as specified in each subsection. The best and second-best numerical results are highlighted in bold and underlined, respectively. All experiments are performed on a PC with an Intel i5-10300H CPU and 16 GB of RAM.

\subsection{Simulation Experiments}\label{Simulation}

\subsubsection{Exact Recovery Under Varying Conditions}\label{experiment1}

We first consider exact recovery of low-rank ($\cL$) and sparse ($\cS$) components on synthetic tensors of size $n \times n \times n$ ($n \in \{100, 200\}$). 
The true tensor $\cL_*$ is generated by $\cL_{*}=\left(\bU_{1},\bU_{2},\bU_{3} \right)\cdot\cG$ with the Tucker rank  $[r, r, r]$ ($r \in \{10, 20, 30\}$),  where $\cG\in \mathbb{R}^{r\times r\times r}$ and $\bU_{i}\in\mathbb{R}^{r\times n},i=1,2,3$ with entries independently sampled from a standard normal distribution. 
The elements of $\cL_{*}$ are then scaled to the range $[-1,1]$.
The true sparse tensor $\cS_*$ is generated by a Bernoulli $\pm 1$ distribution as follows
\begin{equation}\label{Bernoulli}
[\cS_{*}]_{ijk}=
\begin{cases}
1, & \text{w.p. }\frac{\rho_s}{2},\\
0, &\text{w.p. }1-\rho_s,\\
-1, & \text{w.p. }\frac{\rho_s}{2}.
\end{cases}
\end{equation}
with sparsity ratios $\rho_s \in \{0.1, 0.3, 0.5\}$.  
The results are compared with those obtained by RTCUR-FF \cite{cai2024robust}, a CUR-based low-rank approximation method operating on fiber/slice subsets. 
The relative error is used to measure the quality of the recovery results, which is defined as
$$
\text{Rel}_{\widehat{\cL}}=\frac{\|\widehat{\cL}-\cL_{*}\|_{F}}{\|\cL_{*}\|_{F}}, \text{Rel}_{\widehat{\cS}}=\frac{\|\widehat{\cS}-\cS_{*}\|_{F}}{\|\cS_{*}\|_{F}}, 
$$
where $\widehat{\cL}$ and $\widehat{\cS}$ represent the recovered tensors. 
Three key metrics are reported: the relative errors of the recovered low-rank tensor $\widehat{\cL}$ and outlier tensor $\widehat{\cS}$, as well as the computational time in seconds. The comparisons of relative error between our proposed method and RTCUR-FF are shown in Table \ref{correct recovery}. 
Our method demonstrates superior recovery accuracy across all test cases, achieving remarkably low relative errors of less than  $3.0\times 10^{-8}$ for $\cL$ and $4.0\times 10^{-9}$  for $\cS$.
In contrast, RTCUR-FF exhibits higher errors, exceeding $1.0\times 10^{-5}$ for $\cL$ and $1.0\times 10^{-6}$  for $\cS$ when $n=100, \rho=0.5$. 
Moreover, RTCUR-FF fails completely when $n=100, r=30, \rho_s = 0.5$, yielding unreliable results. 
Our proposed method is at least 20 times faster than the RTCUR-FF method according to the CPU running time. This improvement is mainly attributed to the fact that, in RTCUR-FF, a higher rank results in more sampling points, which substantially increases the computational cost.
In summary, the proposed method outperforms RTCUR-FF both in accuracy and computational efficiency, making it a robust and practical method for tensor recovery tasks.

\renewcommand\arraystretch{1.15}
\begin{table}[htbp]
	\centering\footnotesize\tabcolsep=1.mm
\caption{Comparison of recovery results: Relative errors and computational time for RTCUR-FF and our method.}
\label{correct recovery}
	\begin{tabular}{cccccccc}
		\hline
		& \multicolumn{1}{c|}{} & \multicolumn{3}{c||}{RTCUR-FF} & \multicolumn{3}{c}{Ours} \\ \hline
		\multicolumn{1}{c|}{$r$} & \multicolumn{1}{c|}{$\rho$} & Rel$_{\widehat{\cL}}$ & Rel$_{\widehat{\cS}}$ & \multicolumn{1}{c||}{Time(s)} & Rel$_{\widehat{\cL}}$ & Rel$_{\widehat{\cS}}$ & Time(s) \\ \hline
		\multicolumn{8}{c}{$n=100$} \\ \hline
		\multicolumn{1}{c|}{\multirow{3}{*}{10}} & \multicolumn{1}{c|}{0.1} & 6.81e-09 & 8.78e-10 & \multicolumn{1}{c||}{24.29} & \textbf{8.30e-09} & \textbf{1.10e-09} & \textbf{1.17} \\ \cline{2-8} 
		\multicolumn{1}{c|}{} & \multicolumn{1}{c|}{0.3} & 2.20e-08 & 2.43e-09 & \multicolumn{1}{c||}{38.97} & \textbf{1.37e-08} & \textbf{1.59e-09} & \textbf{1.71} \\ \cline{2-8} 
		\multicolumn{1}{c|}{} & \multicolumn{1}{c|}{0.5} & 1.67e-05 & 2.51e-06 & \multicolumn{1}{c||}{76.24} & \textbf{1.94e-08} & \textbf{3.08e-09} & \textbf{2.57} \\ \cline{1-8} 
		\multicolumn{1}{c|}{\multirow{3}{*}{20}} & \multicolumn{1}{c|}{0.1} & 3.07e-09 & 4.29e-10 & \multicolumn{1}{c||}{41.84} & \textbf{4.90e-09} & \textbf{7.75e-10} & \textbf{1.31} \\ \cline{2-8} 
		\multicolumn{1}{c|}{} & \multicolumn{1}{c|}{0.3} & 2.22e-08 & 3.51e-09 & \multicolumn{1}{c||}{69.09} & \textbf{7.49e-09} & \textbf{1.28e-09} & \textbf{2.25} \\ \cline{2-8} 
		\multicolumn{1}{c|}{} & \multicolumn{1}{c|}{0.5} & 1.43e-04 & 2.38e-05 & \multicolumn{1}{c||}{128.98} & \textbf{2.07e-08} & \textbf{3.50e-09} & \textbf{3.56} \\ \cline{1-8} 
		\multicolumn{1}{c|}{\multirow{3}{*}{30}} & \multicolumn{1}{c|}{0.1} & 5.82e-09 & 1.16e-09 & \multicolumn{1}{c||}{59.60} & \textbf{3.92e-09} & \textbf{9.04e-10} & \textbf{2.02} \\ \cline{2-8} 
		\multicolumn{1}{c|}{} & \multicolumn{1}{c|}{0.3} & 1.60e-08 & 2.45e-09 & \multicolumn{1}{c||}{105.33} & \textbf{1.10e-08} & \textbf{1.84e-09} & \textbf{2.98} \\ \cline{2-8} 
		\multicolumn{1}{c|}{} & \multicolumn{1}{c|}{0.5} & 1.01e+00 & 2.11e-01 & \multicolumn{1}{c||}{149.72} & \textbf{1.66e-08} & \textbf{3.69e-09} & \textbf{5.16} \\ \hline
		\multicolumn{8}{c}{$n=200$} \\ \hline
		\multicolumn{1}{c|}{\multirow{3}{*}{10}} & \multicolumn{1}{c|}{0.1} & 5.55e-09 & 5.14e-10 & \multicolumn{1}{c||}{134.58} & \textbf{5.36e-09} & \textbf{5.39e-10} & \textbf{6.62} \\ \cline{2-8} 
		\multicolumn{1}{c|}{} & \multicolumn{1}{c|}{0.3} & 2.66e-08 & 2.80e-09 & \multicolumn{1}{c||}{285.37} & \textbf{1.55e-08} & \textbf{1.65e-09} & \textbf{9.31} \\ \cline{2-8} 
		\multicolumn{1}{c|}{} & \multicolumn{1}{c|}{0.5} & 2.20e-03 & 2.99e-04 & \multicolumn{1}{c||}{420.92} & \textbf{2.05e-08} & \textbf{2.94e-09} & \textbf{13.81} \\ \cline{1-8} 
		\multicolumn{1}{c|}{\multirow{3}{*}{20}} & \multicolumn{1}{c|}{0.1} & 3.08e-09 & 3.14e-10 & \multicolumn{1}{c||}{303.96} & \textbf{1.22e-08} & \textbf{1.29e-09} & \textbf{7.27} \\ \cline{2-8} 
		\multicolumn{1}{c|}{} & \multicolumn{1}{c|}{0.3} & 2.11e-08 & 4.29e-09 & \multicolumn{1}{c||}{502.04} & \textbf{2.21e-08} & \textbf{2.50e-09} & \textbf{10.61} \\ \cline{2-8} 
		\multicolumn{1}{c|}{} & \multicolumn{1}{c|}{0.5} & 5.16e-08 & 6.09e-09 & \multicolumn{1}{c||}{810.82} & \textbf{1.81e-08} & \textbf{3.08e-09} & \textbf{16.12} \\ \cline{1-8} 
		\multicolumn{1}{c|}{\multirow{3}{*}{30}} & \multicolumn{1}{c|}{0.1} & 2.08e-09 & 2.81e-10 & \multicolumn{1}{c||}{449.91} & \textbf{6.30e-09} & \textbf{9.15e-10} & \textbf{8.75} \\ \cline{2-8} 
		\multicolumn{1}{c|}{} & \multicolumn{1}{c|}{0.3} & 1.88e-08 & 2.37e-09 & \multicolumn{1}{c||}{721.28} & \textbf{1.19e-08} & \textbf{1.59e-09} & \textbf{12.15} \\ \cline{2-8} 
		\multicolumn{1}{c|}{} & \multicolumn{1}{c|}{0.5} & 5.55e-08 & 8.56e-09 & \multicolumn{1}{c||}{992.31} & \textbf{1.83e-08} & \textbf{3.26e-09} & \textbf{18.68} \\ \hline
	\end{tabular}
\end{table}

\subsubsection{Phase Transition}\label{experiment2}

In this part, we further evaluate the recovery results with varying Tucker ranks of $\cL_{*}$ and different levels of sparsity in  $\cS_{*}$. For simplicity, we consider tensors of size $50\times50\times50$. First, we generate a low-rank tensor $\cL_{*}$ with Tucker rank $[r, r, r]$, using the same procedure described in Section \ref{experiment1}. For the sparse component, we consider two cases: random signs and coherent signs. In the case of random signs, we generate a sparse tensor that satisfies the Bernoulli $\pm1$ distribution as shown in \eqref{Bernoulli}. In the case of coherent signs, we generate a sparse tensor defined as
\begin{equation*}
[\cS_{*}]_{ijk}=
\begin{cases}
\operatorname{sign}\left([\cL_{*}]_{ijk}\right), & \text{w.p. }\rho_s,\\
0, &\text{w.p. }1-\rho_s,
\end{cases}
\end{equation*}
where $\operatorname{sign}(\cdot)$ is the sign function. To investigate the phase transition in Tucker rank and sparsity, we vary $r \in [2:1:50]$ and $\rho_s \in [0.01:0.03:1]$, conducting 10 experiments for each pair of $(r, \rho_s)$. An experiment is considered successful if the recovered tensor $\widehat{\cL}$ satisfies $\frac{\|\widehat{\cL}-\cL_{*}\|_{F}}{\|\cL_{*}\|_{F}}\le 10^{-3}$. 
Fig. \ref{Phase transition} shows the success rate for each pair of $(r,\rho_s)$ using our proposed method and RTCUR-FF \cite{cai2024robust} (with white indicating 100$\%$ success and black indicating 0$\%$ success). As shown, our method successfully recovers $\widehat{\cL}$ even for highly ill-conditioned problems, while RTCUR-FF struggles to achieve correct recovery when $\rho_s>0.4$. This further demonstrates the superiority of our model.

\begin{figure}[htbp]
\renewcommand{\arraystretch}{0.5}
\setlength\tabcolsep{0.3pt}
	\centering
	\begin{tabular}{cc}
		\includegraphics[width=0.5\linewidth]{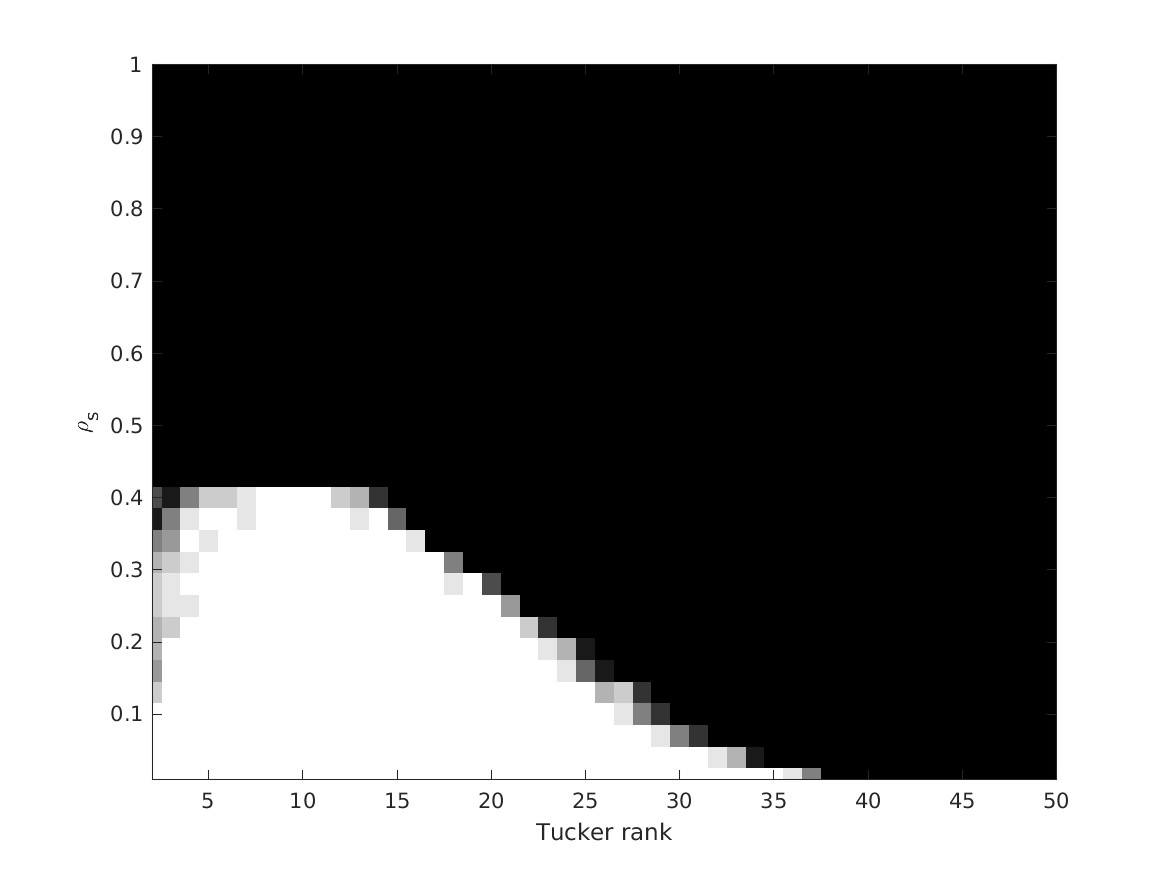} & 
		\includegraphics[width=0.5\linewidth]{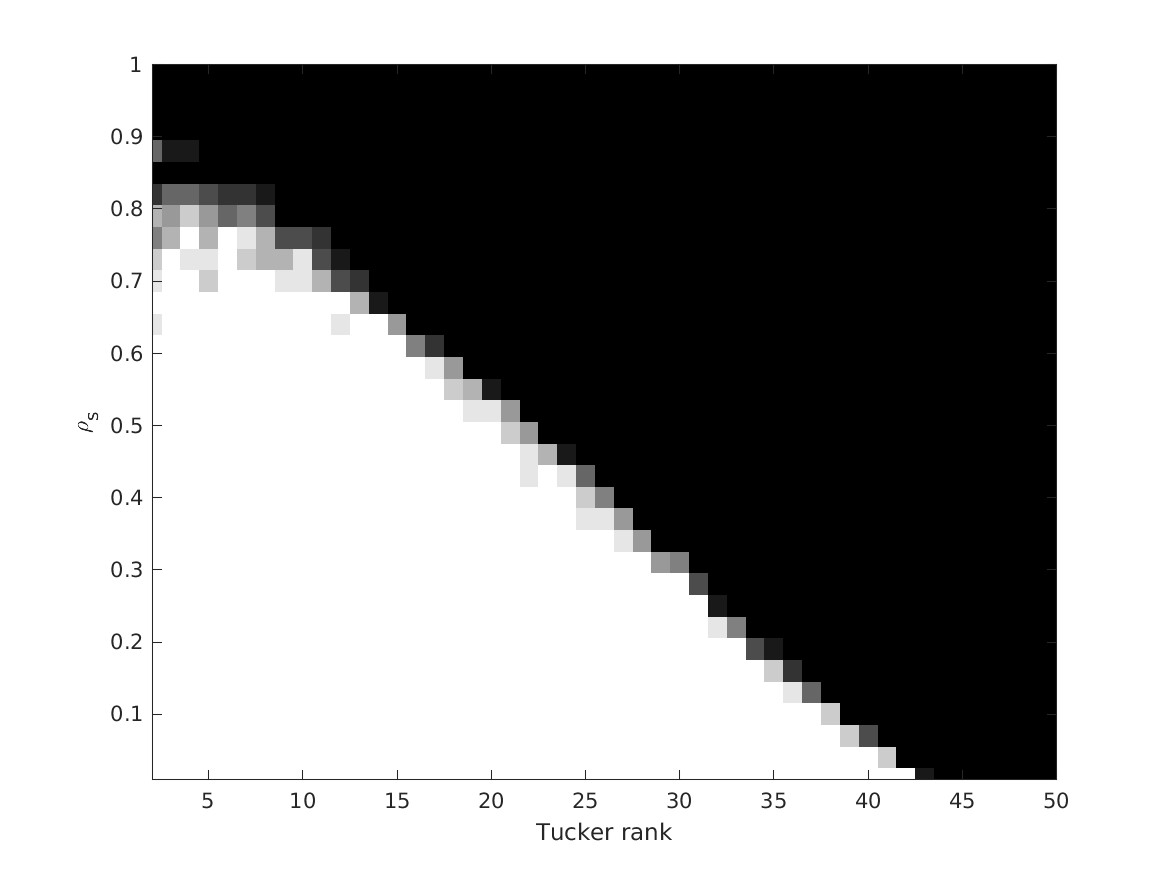}   \\
		\footnotesize (a) RTCUR-FF, random signs & \footnotesize (b) Ours, random signs \\
\includegraphics[width=0.5\linewidth]{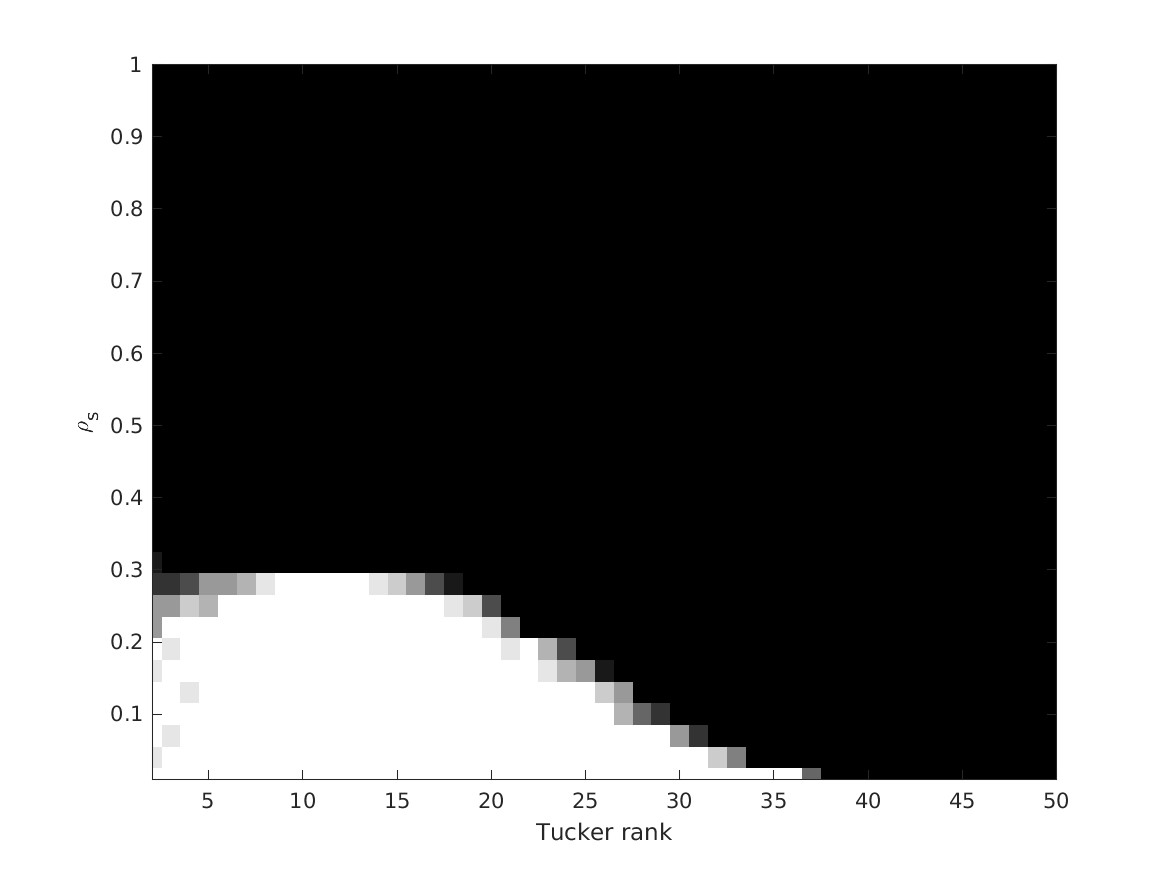} & \includegraphics[width=0.5\linewidth]{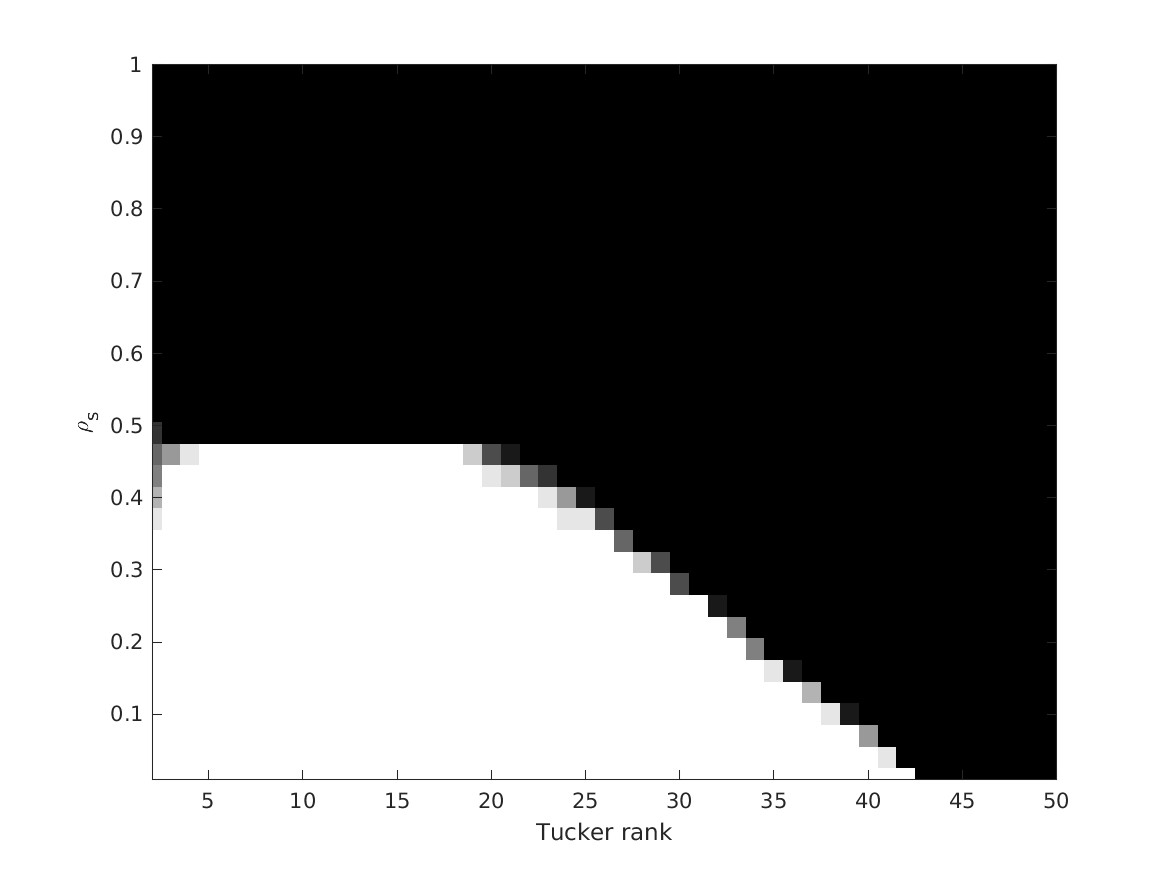} \\
\footnotesize (c) RTCUR-FF, coherent signs & \footnotesize (d) Ours, coherent signs
	\end{tabular}
	\caption{Correct recovery for different levels of Tucker rank and sparsity. Fraction of correct recoveries across 10 trials, as a function of Tucker rank ($x$-axis) and sparsity of $\cS_0$ ($y$-axis).}
	\label{Phase transition}
\end{figure}

\subsubsection{Sensitivity Analysis with Respect to $\lambda$}

To investigate the impact of the model parameter $\lambda$, we conduct a set of experiments on a synthetic tensor of size $100\times 100\times 100$, with the Tucker rank fixed at $[10, 10, 10]$. The noise scenario follows the random signs setting described in \eqref{Bernoulli}, with $\rho_s\in\{0.3, 0.6\}$.

We test low-rank recovery under different values of $\lambda=0.1,0.2,0.5,1,2,3$. As illustrated in Fig. \ref{Parameter lambda}, we plot the relative error between the reconstructed low-rank tensor $\cL_t$ at the $t$-th iteration and the ground truth $\cL_*$. For the case of $\rho_s = 0.3$, all tested values of $\lambda$ successfully recover the low-rank structure, and a larger $\lambda$ generally leads to faster convergence. In the case of $\rho_s = 0.6$, all tested values of $\lambda$ still achieve low final relative errors. Moreover, while larger $\lambda$ values accelerate convergence, smaller ones (e.g., $\lambda=0.2$ and $\lambda=0.5$) yield slightly more accurate final recovery. These results indicate that our model exhibits low sensitivity to the choice of $\lambda$. Considering both convergence speed and practical performance, we select $\lambda=1$ as the default setting in subsequent experiments.

\begin{figure}[htbp]
\renewcommand{\arraystretch}{0.5}
\setlength\tabcolsep{0.3pt}
	\centering
	\begin{tabular}{cc}
		\includegraphics[width=0.5\linewidth]{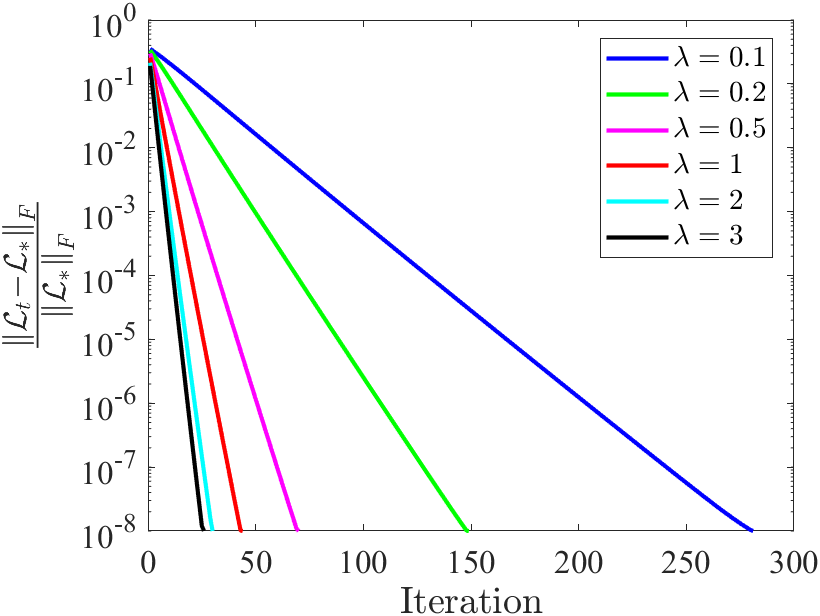} & 
		\includegraphics[width=0.5\linewidth]{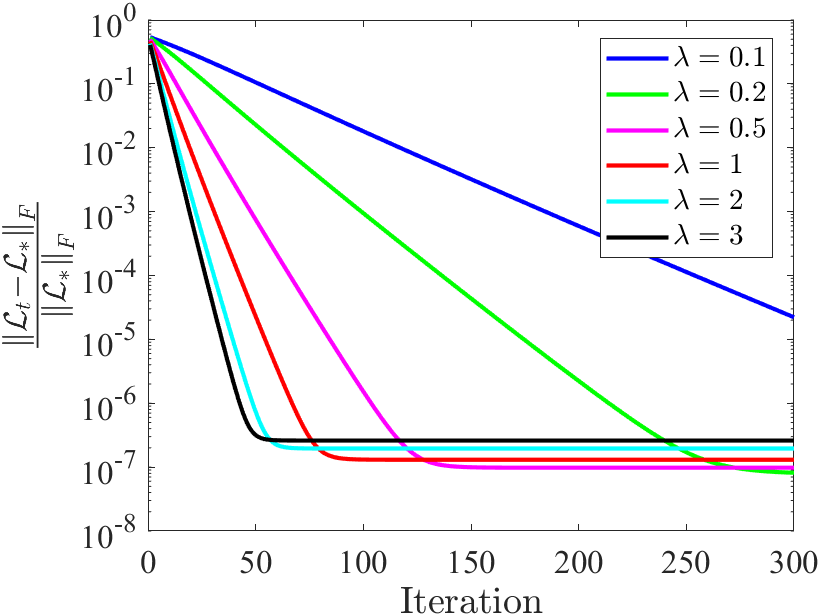}
	\end{tabular}
	\caption{Relative error of $\cL_t$ versus iteration number for different values of $\lambda$. The left and right plots correspond to $\rho=0.3$ and $\rho=0.6$, respectively.}
	\label{Parameter lambda}
\end{figure}

\subsection{Applications}\label{Applications}

In this subsection, the proposed method is applied to real-world tasks, such as face image denoising, background subtraction, and hyperspectral image denoising. The methods compared include TNN \cite{lu2019tensor}, KBR \cite{xie2017kronecker}, LRTV \cite{he2015total}, ETRPCA \cite{gao2020enhanced}, FTTNN \cite{qiu2022efficient}, t-CTV \cite{wang2023guaranteed} and RTCUR-FF \cite{cai2024robust}. The parameters of all baseline methods are set based on the authors' recommendations or fine-tuned to achieve optimal results.

\subsubsection{Face Image Denoising Under Variable Illumination}\label{experiment3}

We begin by addressing the problem of recovering facial features from images affected by varying lighting, facial expressions, and noise. In this approach, the face is modeled as a low-rank component, while shadows and noise are treated as sparse outliers. This separation allows us to isolate the true facial structure from disruptive artifacts. We demonstrate the method using the Yale B dataset \cite{Georghiades2001From}, which consists of 64 face images of $192\times168$, forming a $192\times168\times64$ input tensor. To simulate realistic distortions, we add 20\% random noise, thereby increasing the difficulty of the task. 

For our method, we set the Tucker rank to $[30,30,1]$ to capture the consistent face structure across varying illumination conditions, and the number of iterations to 10. The visual results of all methods are displayed in Fig. \ref{face denoising}. The compared methods fail to fully restore the face, especially with the shadow around the nose not being completely removed. In contrast, our proposed method excels by effectively eliminating both shadows and random noise, resulting in a clearer depiction of facial features and more uniform lighting across the image, thus offering a significantly better visual outcome. Notably, our method is capable of completing the face restoration \textbf{within just one second}.

\begin{figure*}[htbp]
\renewcommand{\arraystretch}{0.45}
\setlength\tabcolsep{0.5pt}
	\centering
	\begin{tabular}{ccccccccccc}
		\includegraphics[width=0.09\linewidth]{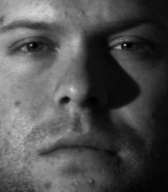} & 
		\includegraphics[width=0.09\linewidth]{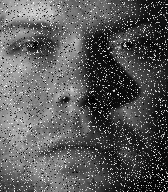} & \includegraphics[width=0.09\linewidth]{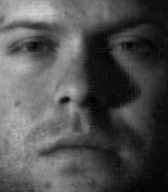} & \includegraphics[width=0.09\linewidth]{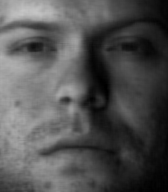}& \includegraphics[width=0.09\linewidth]{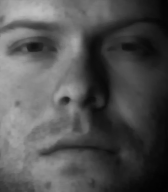}& \includegraphics[width=0.09\linewidth]{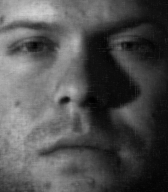}&\includegraphics[width=0.09\linewidth]{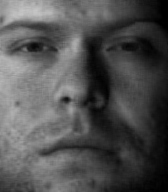}&\includegraphics[width=0.09\linewidth]{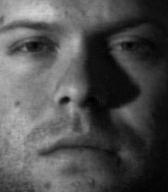}& \includegraphics[width=0.09\linewidth]{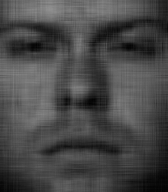}& \includegraphics[width=0.09\linewidth]{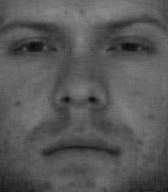} \\
		\includegraphics[width=0.09\linewidth]{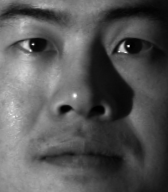} & 
		\includegraphics[width=0.09\linewidth]{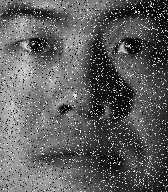} & \includegraphics[width=0.09\linewidth]{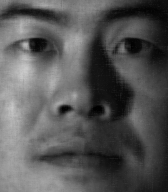} & \includegraphics[width=0.09\linewidth]{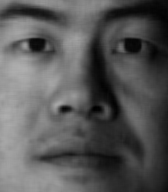}& \includegraphics[width=0.09\linewidth]{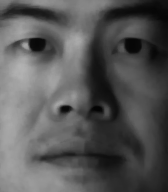}& \includegraphics[width=0.09\linewidth]{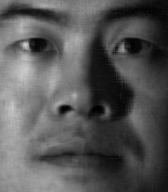}&\includegraphics[width=0.09\linewidth]{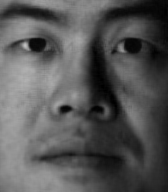}&
		\includegraphics[width=0.09\linewidth]{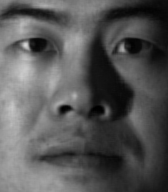}& \includegraphics[width=0.09\linewidth]{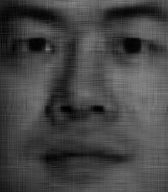}& \includegraphics[width=0.09\linewidth]{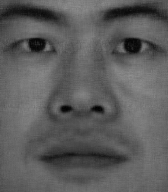} \\
		\includegraphics[width=0.09\linewidth]{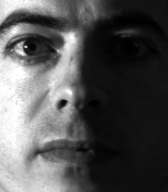} & 
		\includegraphics[width=0.09\linewidth]{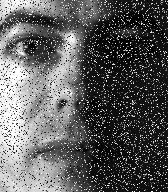} & \includegraphics[width=0.09\linewidth]{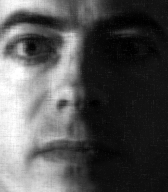} & \includegraphics[width=0.09\linewidth]{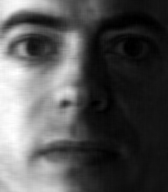}& \includegraphics[width=0.09\linewidth]{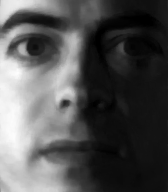}& \includegraphics[width=0.09\linewidth]{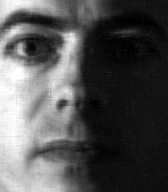}&\includegraphics[width=0.09\linewidth]{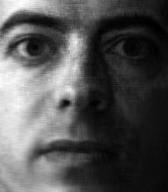}&
		\includegraphics[width=0.09\linewidth]{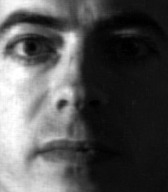}& \includegraphics[width=0.09\linewidth]{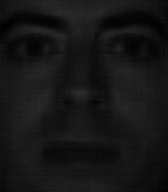}& \includegraphics[width=0.09\linewidth]{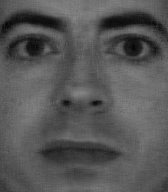} \\
		\footnotesize Original & \footnotesize Noisy & \footnotesize TNN & \footnotesize KBR & \footnotesize LRTV & \footnotesize ETRPCA & \footnotesize FTTNN & \footnotesize t-CTV & \footnotesize RTCUT-FF & \footnotesize Ours \\
        \footnotesize{--} & \footnotesize{--} & \footnotesize{36.17s} & \footnotesize{55.02s} & \footnotesize{9.04s} & \footnotesize{114.04s} & \footnotesize{8.79s} & \footnotesize{90.48s} & \footnotesize{31.19s} & \footnotesize{\textbf{0.79s}}
	\end{tabular}
	\caption{Noise and shadows removal from face images of subject 1 (row I), subject 2 (row II), and subject 6 (row III). The values below each method indicate the average running time in seconds (s).}
	\label{face denoising}
\end{figure*}

\subsubsection{Background Subtraction}\label{experiment5}

This section demonstrates the application of our proposed algorithm for background subtraction, aiming to separate foreground objects from the background in a video sequence. In these videos, the background can be modeled as a low-rank tensor, while the foreground is sparse and exhibits significant changes, making it well-suited for sparse outlier detection. For our experiments, we use four videos from the CDnet dataset \cite{goyette2012changedetection}: ``blizzard'', ``office'', ``skating'', and ``snowFall''. We select 300 frames from each, and downsample them to $160\times 180\times 3$.  Since not all frames contain ground truth data, we select 20 frames with labeled ground truth \cite{zhou2019bayesian}.

For evaluation, we focus on comparing the sparse component (foreground) $\widehat{\cS}$ with the ground truth. Following \cite{yang2022nonconvex}, we apply the hard thresholding function to binarize each slice of $\widehat{\cS}$, using the standard deviation of each slice as the threshold. Then, $5 \times 5$ median filter is applied to the binarized foreground tensor. We treat the processed foreground tensor and the ground truth tensor $\cS_{\text{gt}}$ as a binary classification problem. The performance is evaluated using precision, recall, and F-measure metrics:
\begin{align*}
	&\text{Precision}=\frac{\mathrm{TP}}{\mathrm{TP}+\mathrm{FP}}, \text{Recall}=\frac{\mathrm{TP}}{\mathrm{TP}+\mathrm{FN}}, \\
	&\text {F-measure}=2 \frac{\text {Precision} \cdot \text {Recall}}{\text {Precision}+ \text{Recall}}.
\end{align*}
Here, TP (True Positive) represents the number of correctly detected foreground pixels; FP (False Positive) represents pixels incorrectly detected as foreground; and FN (False Negative) represents foreground pixels that were missed.

For our model, the input videos are reshaped into a tensor of size $[\text{height}\times\text{width},\text{channel},\text{frame}]$. We set the Tucker rank to $\br = [3,3,1]$ and the maximum number of iterations to $T = 40$. The Tucker rank $\br = [3, 3, 1]$ is chosen to capture the low-rank background, with $r_1 = 3$ and $r_2 = 3$ for spatial and channel redundancy, and $r_3 = 1$ for the static frame structure. In the oracle case, we set $\cW_{ijk} = 0$ when $\cS_{\text{gt}} \neq 0$ and set the remaining elements to 1. The background subtraction results are shown in Fig. \ref{fig:background}, and the numerical results are summarized in Table \ref{tab:background}. It can be observed that most of the compared methods fail to fully separate the foreground and background in certain video frames. Note that t-CTV \cite{wang2023guaranteed} cannot perform background subtraction as its model is primarily designed for tensor recovery, and therefore it is not considered a comparison method in this experiment. While RTCUR-FF achieves computational efficiency through its sampling-based strategy with a low Tucker rank setting ($[3, 3, 1]$), this approach results in incomplete background separation due to small size sampling, as demonstrated in Fig. \ref{fig:background} (``skating'' video). 
Our method maintains superior separation quality with stable performance, outperforming all other baseline methods (except RTCUR-FF) by at least 2.5 times in speed. 
We remark,  when replacing Tucker decomposition with CUR decomposition in our framework, the computational efficiency becomes comparable to RTCUR-FF while preserving the robustness advantages of our approach. 

\begin{figure*}[htbp]
	\renewcommand{\arraystretch}{0.5}
\setlength\tabcolsep{0.5pt}
\centering
	\begin{tabular}{ccccccccc}
		\includegraphics[width=0.1\linewidth]{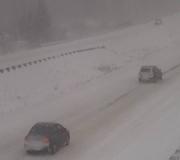} & \includegraphics[width=0.1\linewidth]{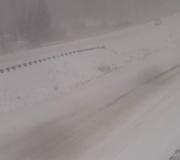} & \includegraphics[width=0.1\linewidth]{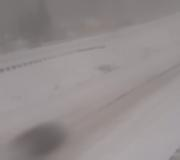}& \includegraphics[width=0.1\linewidth]{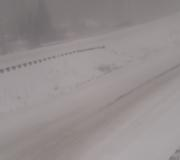}& \includegraphics[width=0.1\linewidth]{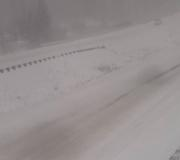}& \includegraphics[width=0.1\linewidth]{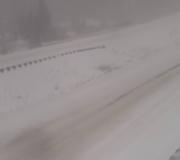}
        & \includegraphics[width=0.1\linewidth]{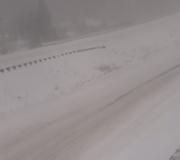}& \includegraphics[width=0.1\linewidth]{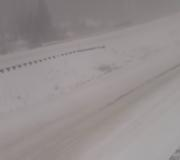}& \includegraphics[width=0.1\linewidth]{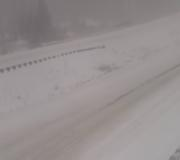} \\
		\includegraphics[width=0.1\linewidth]{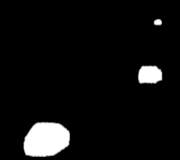} & \includegraphics[width=0.1\linewidth]{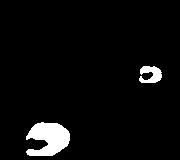} & \includegraphics[width=0.1\linewidth]{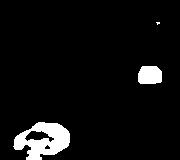}& \includegraphics[width=0.1\linewidth]{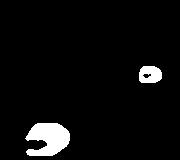}& \includegraphics[width=0.1\linewidth]{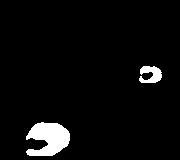}& \includegraphics[width=0.1\linewidth]{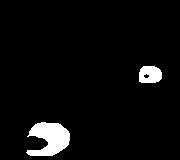}& \includegraphics[width=0.1\linewidth]{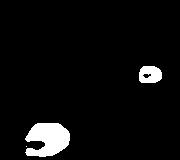}& \includegraphics[width=0.1\linewidth]{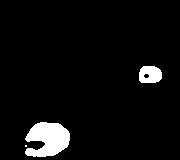}& \includegraphics[width=0.1\linewidth]{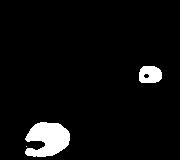} \\
		\includegraphics[width=0.1\linewidth]{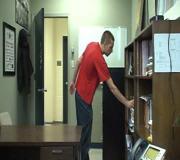} & \includegraphics[width=0.1\linewidth]{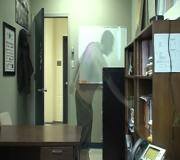} & \includegraphics[width=0.1\linewidth]{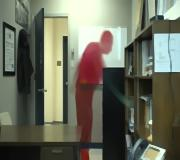}& \includegraphics[width=0.1\linewidth]{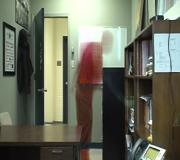}& \includegraphics[width=0.1\linewidth]{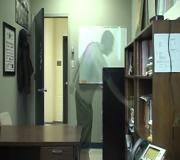}& \includegraphics[width=0.1\linewidth]{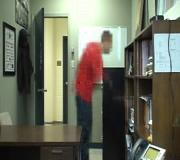}& \includegraphics[width=0.1\linewidth]{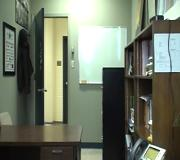}& \includegraphics[width=0.1\linewidth]{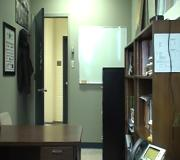}& \includegraphics[width=0.1\linewidth]{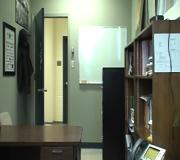} \\
		\includegraphics[width=0.1\linewidth]{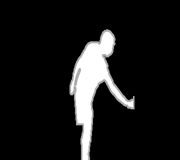} & \includegraphics[width=0.1\linewidth]{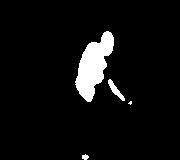} & \includegraphics[width=0.1\linewidth]{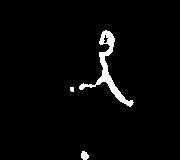}& \includegraphics[width=0.1\linewidth]{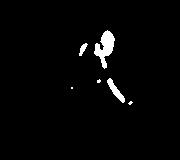}& \includegraphics[width=0.1\linewidth]{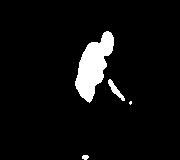}& \includegraphics[width=0.1\linewidth]{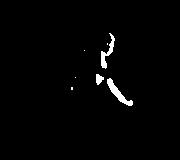}& \includegraphics[width=0.1\linewidth]{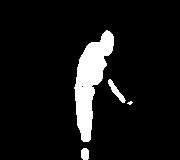}& \includegraphics[width=0.1\linewidth]{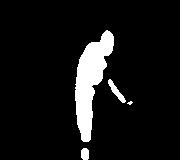}& \includegraphics[width=0.1\linewidth]{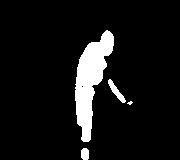} \\
		\includegraphics[width=0.1\linewidth]{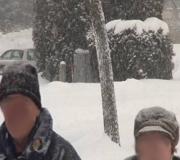} & \includegraphics[width=0.1\linewidth]{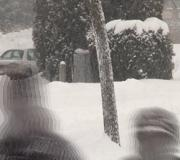} & \includegraphics[width=0.1\linewidth]{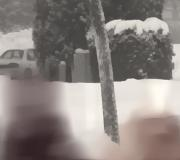}& \includegraphics[width=0.1\linewidth]{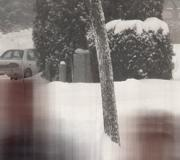}& \includegraphics[width=0.1\linewidth]{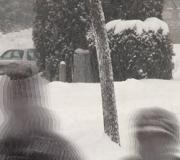}& \includegraphics[width=0.1\linewidth]{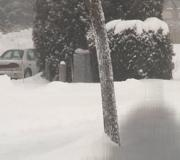}& \includegraphics[width=0.1\linewidth]{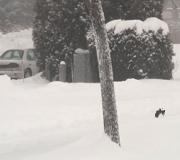}& \includegraphics[width=0.1\linewidth]{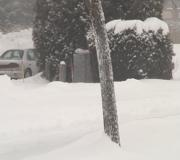}& \includegraphics[width=0.1\linewidth]{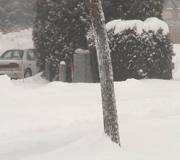} \\
		\includegraphics[width=0.1\linewidth]{figures/Background_subtraction_results/skating_gt} & \includegraphics[width=0.1\linewidth]{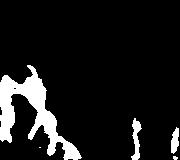} & \includegraphics[width=0.1\linewidth]{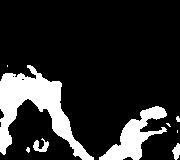}& \includegraphics[width=0.1\linewidth]{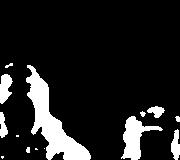}& \includegraphics[width=0.1\linewidth]{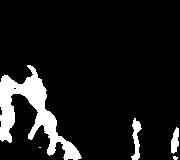}& \includegraphics[width=0.1\linewidth]{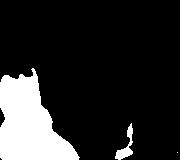}& \includegraphics[width=0.1\linewidth]{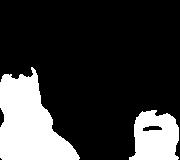}& \includegraphics[width=0.1\linewidth]{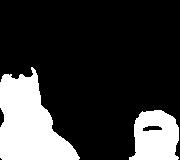}& \includegraphics[width=0.1\linewidth]{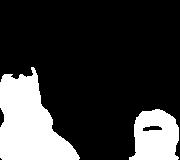} \\
		\includegraphics[width=0.1\linewidth]{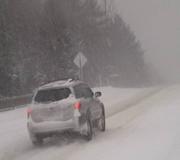} & \includegraphics[width=0.1\linewidth]{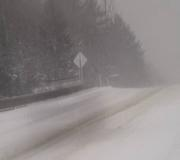} & \includegraphics[width=0.1\linewidth]{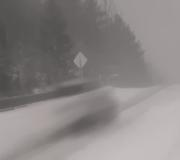}& \includegraphics[width=0.1\linewidth]{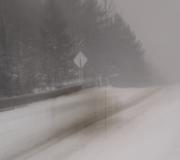}& \includegraphics[width=0.1\linewidth]{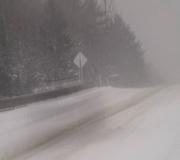}& \includegraphics[width=0.1\linewidth]{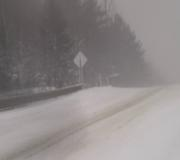}& \includegraphics[width=0.1\linewidth]{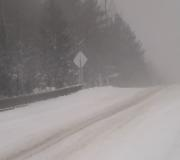}& \includegraphics[width=0.1\linewidth]{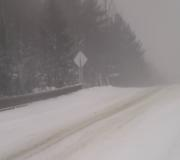}& \includegraphics[width=0.1\linewidth]{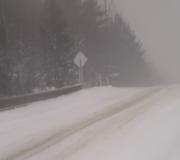} \\
		\includegraphics[width=0.1\linewidth]{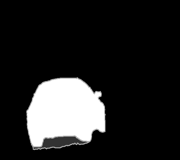} & \includegraphics[width=0.1\linewidth]{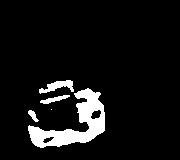} & \includegraphics[width=0.1\linewidth]{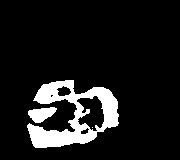}& \includegraphics[width=0.1\linewidth]{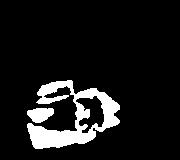}& \includegraphics[width=0.1\linewidth]{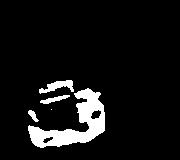}& \includegraphics[width=0.1\linewidth]{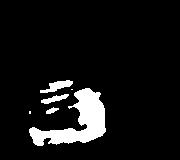}& \includegraphics[width=0.1\linewidth]{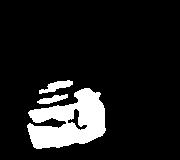}& \includegraphics[width=0.1\linewidth]{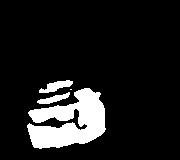}& \includegraphics[width=0.1\linewidth]{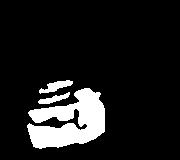} \\
		\footnotesize Original  & \footnotesize TNN & \footnotesize LRTV & \footnotesize KBR & \footnotesize ETRPCA & \footnotesize FTTNN & \footnotesize RTCUR-FF & \footnotesize Ours & \footnotesize Oracle
	\end{tabular}
	\caption{Visual comparison of background and foreground extraction from four videos. The first column shows the original data and ground truth, while the subsequent columns display the background and foreground extracted by different methods. Each two rows from top to bottom correspond to videos: ``blizzard'', ``office'', ``skating'', and ``snowFall'', respectively.}
 \label{fig:background}
\end{figure*}
 
\begin{table*}[htbp]
	\centering\footnotesize\tabcolsep=1.mm
	\caption{Quantitative results for background subtraction on CDnet dataset. %The Tucker rank is set to $[3,3,1]$ to ensure that the background remains unchanged in each frame.
    }
\label{tab:background}
\begin{tabular}{l|c|c|c|c|c|c|c|c||c}
\hline
Videos                    & Metrics     & TNN    & LRTV   & KBR    & ETRPCA & FTTNN & RTCUR-FF  & Ours   & Oracle \\ \hline
\multirow{4}{*}{blizzard} 
		& Precision$\uparrow$  & 0.9918 & 0.9394 & 0.9838 & 0.9922 & 0.9716 & 0.9714 & 0.9746 & 0.9741 \\
		& Recall$\uparrow$  & 0.7395 & 0.8273 & 0.8190 & 0.7401 & 0.8328 & 0.8454 & 0.8440 & 0.8433\\
		& \textbf{F-measure$\uparrow$} & 0.8473 & 0.8798 & 0.8939 & 0.8478 & 0.8969 & \underline{0.9040} & \textbf{0.9046} & 0.9040 \\
		& \textbf{Time$\downarrow$} & 235.77 & 129.43 & 496.96 & 555.48 & 144.77 & \textbf{5.01} & \underline{46.71} & 44.50 \\ \hline
\multirow{4}{*}{office} 
		& Precision$\uparrow$ & 0.9960 & 0.9036 & 0.9348 & 0.9966 & 0.9043 & 0.9874 & 0.9875 & 0.9953\\
		& Recall$\uparrow$  & 0.5150 & 0.4517 & 0.4739 & 0.5399 & 0.4249 & 0.7216 & 0.7249 & 0.7209\\
		& \textbf{F-measure$\uparrow$} & 0.6789 & 0.6023 & 0.6289 & 0.7004 & 0.5781 & \underline{0.8338} & \textbf{0.8361} & 0.8362\\
		& \textbf{Time$\downarrow$} & 233.51 & 137.04 & 614.99 & 566.31 & 159.05 & \textbf{5.65} & \underline{48.58} & 44.95 \\ \hline
\multirow{4}{*}{skating} 
		& Precision$\uparrow$ & 0.8033 & 0.8819 & 0.7207 & 0.8033 & 0.9515 & 0.9898 & 0.9971 & 0.9970\\
		& Recall$\uparrow$ & 0.3391 & 0.5829 & 0.4426 & 0.3406 & 0.6130 & 0.7289 & 0.7333 & 0.7355\\
		& \textbf{F-measure$\uparrow$} & 0.4769 & 0.7019 & 0.5484 & 0.4784 & 0.7456 & \underline{0.8396} & \textbf{0.8451} & 0.8465\\
		& \textbf{Time$\downarrow$} & 241.61 & 136.70 & 717.82 & 527.29 & 156.64 & \textbf{4.99} & \underline{49.32} & 47.12 \\ \hline
\multirow{4}{*}{snowFall} 
		& Precision$\uparrow$ & 0.9239 & 
        0.8624 & 0.8885 & 0.9234 & 0.8852 & 0.9032 & 0.9032 & 0.9039\\
		& Recall$\uparrow$ & 0.4559 & 0.6175 & 0.6276 & 0.4572 & 0.6688 & 0.7463 & 0.7491 & 0.7455\\
		& \textbf{F-measure$\uparrow$} & 0.6105 & 0.7197 & 0.7356 & 0.6116 & 0.7619 & \underline{0.8173} & \textbf{0.8190} & 0.8171\\
		& \textbf{Time$\downarrow$} & 244.95 & 137.93 & 583.68 & 516.01 & 153.14 & \textbf{4.87} & \underline{48.61} & 44.56 \\ \hline
\end{tabular}
\end{table*}

% We further examine the role of our outlier-aware weight by comparing it with the oracle weight and the ground truth, as shown in Fig. \ref{Fig:outlier-aware weight}. The results reveal that the outlier-aware weight not only accurately identifies moving people and cars but also successfully detects dynamic snow (evidenced by the scattered points in the outlier-aware weight). This demonstrates that our model can effectively achieve precise separation, even in the presence of interference such as adverse weather conditions.
% \tr{
% To evaluate the performance of our proposed outlier-aware weighting approach, we conducted a comparative analysis against both the oracle weight and ground truth annotations, shown in Fig. \ref{Fig:outlier-aware weight}.
% The left subfigure presents the ground truth binary mask, where pixel values are encoded as follows: white regions (value = 1) denote outlier-contaminated areas, while black regions (value = 0) correspond to outlier-free zones. 
% The oracle weight is shown in the middle subfigure, which is computed using the ground truth annotations, where outlier pixels are assigned a weight of 0, while outlier-free pixels are assigned a weight of 1. 
% The outlier-aware weight $\cW$ obtained by the proposed model is shown in the right subfigure.
% We can observe that the outlier-aware weight is very close to the oracle weight.
% }
To evaluate the performance of our proposed outlier-aware weighting approach, we conducted a comparative analysis against both the oracle weight and ground truth annotations, as illustrated in Fig. \ref{Fig:outlier-aware weight}.
The left subfigure shows the ground truth binary mask, where white regions (value = 1) indicate outlier-contaminated areas, and black regions (value = 0) represent outlier-free zones.
The middle subfigure displays the oracle weight, which is derived from the ground truth annotations: outlier pixels are assigned a weight of 0, while outlier-free pixels are assigned a weight of 1.
The right subfigure depicts the outlier-aware weight produced by our proposed model.
We can observe that the estimated outlier-aware weight closely matches the oracle weight. Notably, the results demonstrate that our method not only accurately identifies moving people and vehicles but also effectively detects dynamic snow, as evidenced by the scattered patterns in the outlier-aware weight.

\begin{figure}[htbp]
\renewcommand{\arraystretch}{0.5}
\setlength\tabcolsep{0.3pt}
	\centering
	\begin{tabular}{ccc}
\includegraphics[width=0.3\linewidth]{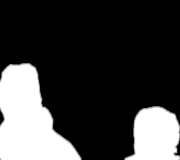} & 
\includegraphics[width=0.3\linewidth]{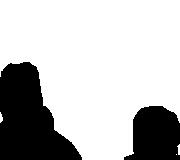} & 	\includegraphics[width=0.3\linewidth]{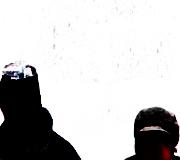}  \\

\includegraphics[width=0.3\linewidth]{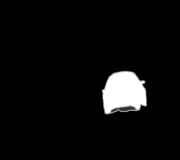} & \includegraphics[width=0.3\linewidth]{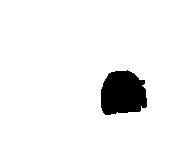}  & 
		
\includegraphics[width=0.3\linewidth]{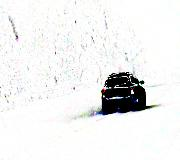}\\
		\footnotesize (a) Ground truth  & \footnotesize (b) Oracle &  \footnotesize (c) Outlier-aware  \\
	\end{tabular}
	\caption{Comparison of weight tensors in background subtraction: (a) Ground truth binary mask (white for outlier, black for clean). (b) Oracle weight (black for outlier, white for clean). (c) Outlier-aware weight estimated by our model.}
	\label{Fig:outlier-aware weight}
\end{figure}

\subsubsection{Hyperspectral Image Denoising}\label{experiment4}

We perform denoising experiments on hyperspectral images (HSIs) from the Indian Pines and Pavia University (PaviaU) datasets, with sizes of $145 \times 145 \times 220$ and $610 \times 340 \times 103$, respectively. To evaluate the quality of denoising, we adopt several widely used image quality metrics, including the peak signal-to-noise ratio (PSNR), the structural similarity index (SSIM) \cite{wang2004image}, which is computed by averaging across all spectral bands, and the erreur relative globale adimensionnelle de synthèse (ERGAS) \cite{wald2002data}.

To evaluate the performance of our proposed weight \eqref{eq:W_impulsive} in handling impulsive noise, we first conduct a denoising experiment on HSIs contaminated by 30\% salt-and-pepper noise. In this case, we assign the weight $\cW_{ijk}=0$ if $\cX_{ijk}=0$ or 1, and $\cW_{ijk}=1$ otherwise. To ensure robust recovery of the underlying HSI structure, we set the Tucker rank to $\br=\left[\lceil0.7n_{1}\rceil,\lceil0.7n_{2}\rceil,\lceil0.05n_{3}\rceil\right]$, where $\lceil \cdot \rceil$ denotes the ceiling function. The higher ranks along the spatial dimensions preserve the detailed structural information of the image, while the smaller rank along the spectral dimension reflects the inherent low-rank property of the spectral signatures under known noise conditions. We perform 80 iterations to achieve improved recovery results. For the oracle case, the tensor weight $\cW$ is pre-defined via \eqref{eq:W}, using the same parameter settings as described above.

Next, we evaluate the model's performance on HSIs contaminated with 30\% random noise. In this case, we use an exponential outlier-aware weight tensor to identify and detect the noise. To further increase task complexity, we introduce stripe noise across bands 1 to 60 of the HSIs. For each band, we generate 20 to 40 columns of stripe noise with intensity values uniformly distributed in the range $[-0.25, 0.25]$. This noise is randomly distributed across columns, with intensity remaining constant within each stripe. Given the increased complexity and detection uncertainty, we adjust the Tucker rank to $\br=\left[\lceil0.7n_{1}\rceil,\lceil0.7n_{2}\rceil,\lceil0.02n_{3}\rceil\right]$, where the spectral rank is reduced to impose a stronger low-rank constraint. This adjustment helps compensate for less precise noise localization and aids in robust recovery.

The denoising numerical results are shown in Table \ref{tab:HSI}, and the visual results are provided in Fig. \ref{fig:HSI}. In the salt-and-pepper noise scenario, we used \eqref{eq:W_impulsive} to determine the weight tensor, achieving results close to the oracle case, significantly outperforming the comparison methods. In the two tested HSIs, the PSNR is more than 2 dB higher than that of other methods. In the remaining two scenarios, our method also shows improvements of at least 0.5 dB over others. Besides, for stripe noise, many methods fail to remove the noise effectively, while our method maintains good performance. Among all the compared methods, our approach also incurs the lowest computational cost.

\begin{figure*}[htbp]
\renewcommand{\arraystretch}{0.3}
\setlength\tabcolsep{0.3pt}
	\centering
	\begin{tabular}{cccccccccc}
		\includegraphics[width=0.09\linewidth]{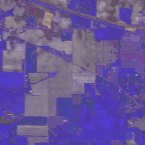} & 
		\includegraphics[width=0.09\linewidth]{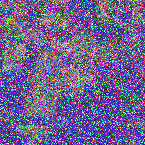} & \includegraphics[width=0.09\linewidth]{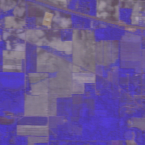} & \includegraphics[width=0.09\linewidth]{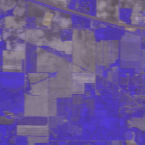}& \includegraphics[width=0.09\linewidth]{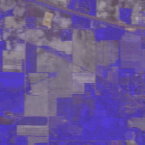}& \includegraphics[width=0.09\linewidth]{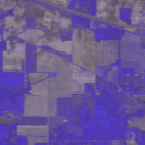}& \includegraphics[width=0.09\linewidth]{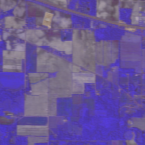}& \includegraphics[width=0.09\linewidth]{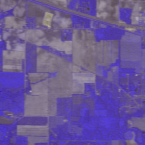}& \includegraphics[width=0.09\linewidth]{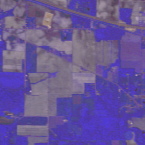}& \includegraphics[width=0.09\linewidth]{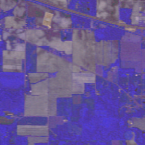} \\
        \includegraphics[width=0.09\linewidth]{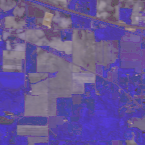} & 
		\includegraphics[width=0.09\linewidth]{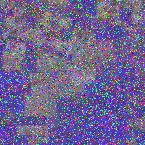} & \includegraphics[width=0.09\linewidth]{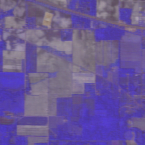} & \includegraphics[width=0.09\linewidth]{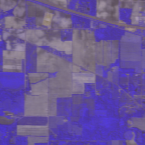}& \includegraphics[width=0.09\linewidth]{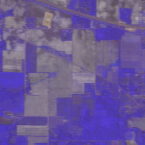}& \includegraphics[width=0.09\linewidth]{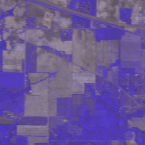}& \includegraphics[width=0.09\linewidth]{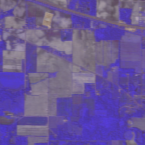}& \includegraphics[width=0.09\linewidth]{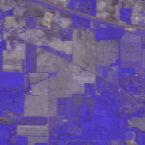}& \includegraphics[width=0.09\linewidth]{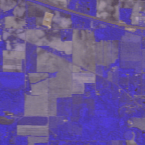}& \includegraphics[width=0.09\linewidth]{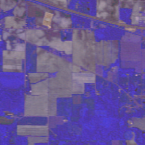} \\
	    \includegraphics[width=0.09\linewidth]{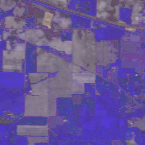}	&
		\includegraphics[width=0.09\linewidth]{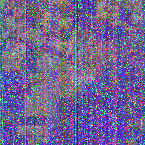} & \includegraphics[width=0.09\linewidth]{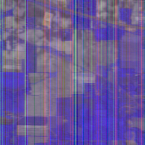} & \includegraphics[width=0.09\linewidth]{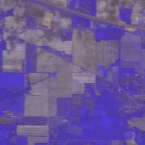}& \includegraphics[width=0.09\linewidth]{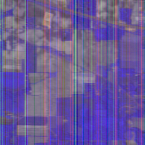}& \includegraphics[width=0.09\linewidth]{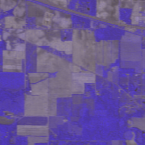}& \includegraphics[width=0.09\linewidth]{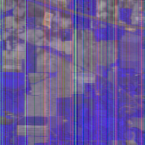}& \includegraphics[width=0.09\linewidth]{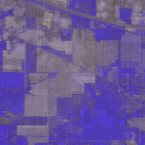}& \includegraphics[width=0.09\linewidth]{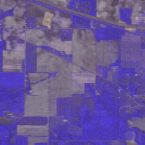}& \includegraphics[width=0.09\linewidth]{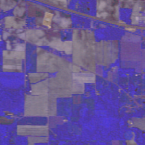} \\
		\includegraphics[width=0.09\linewidth]{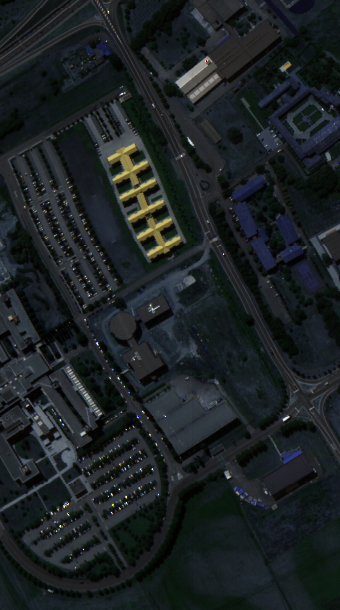} & 
		\includegraphics[width=0.09\linewidth]{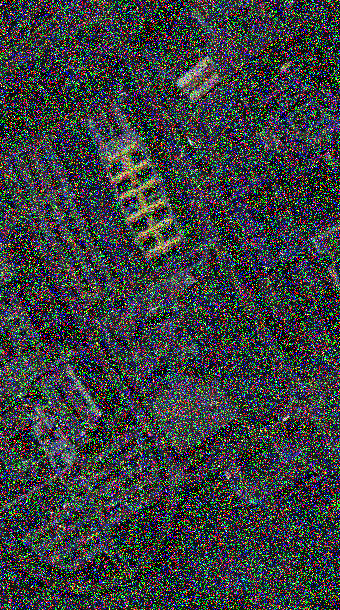} & \includegraphics[width=0.09\linewidth]{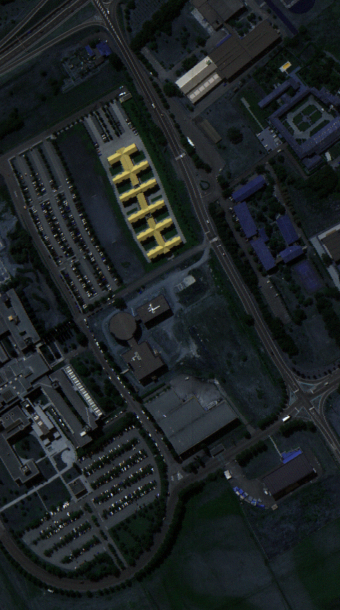} & \includegraphics[width=0.09\linewidth]{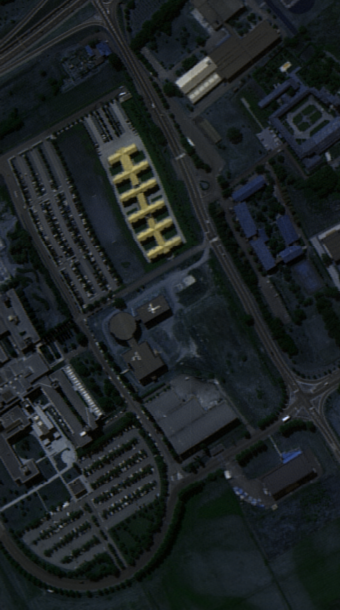}& \includegraphics[width=0.09\linewidth]{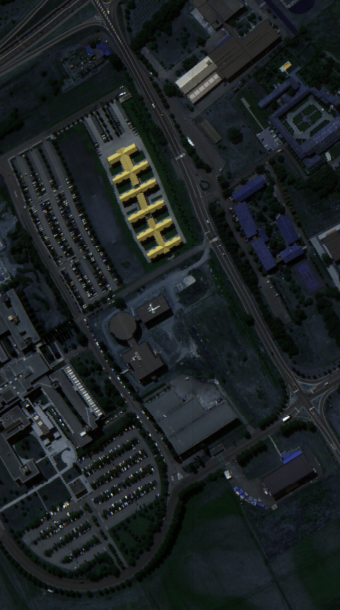}& \includegraphics[width=0.09\linewidth]{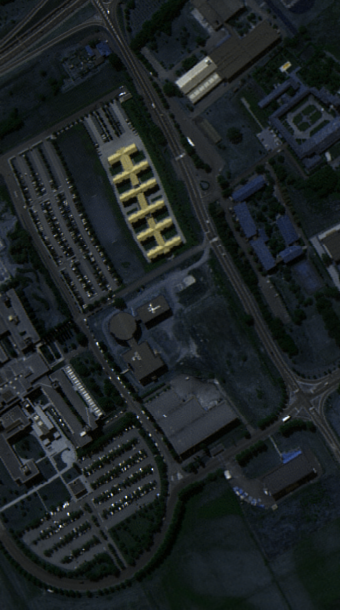}& \includegraphics[width=0.09\linewidth]{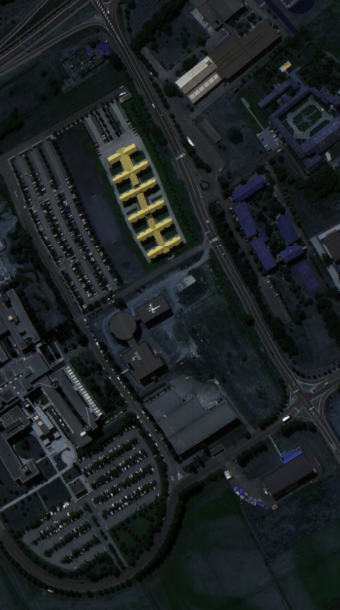}& \includegraphics[width=0.09\linewidth]{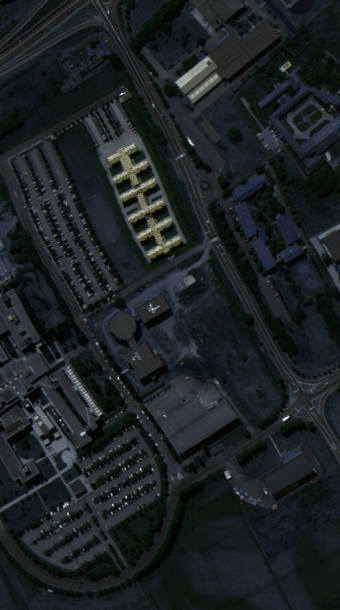}& \includegraphics[width=0.09\linewidth]{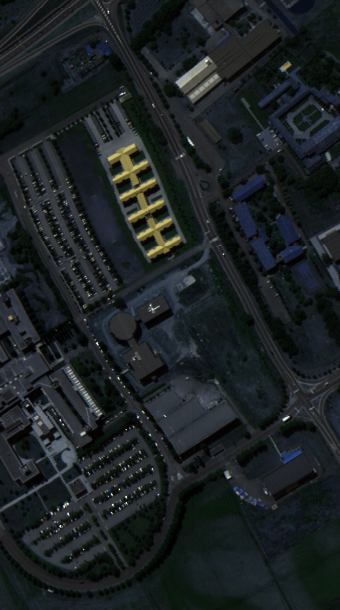}& \includegraphics[width=0.09\linewidth]{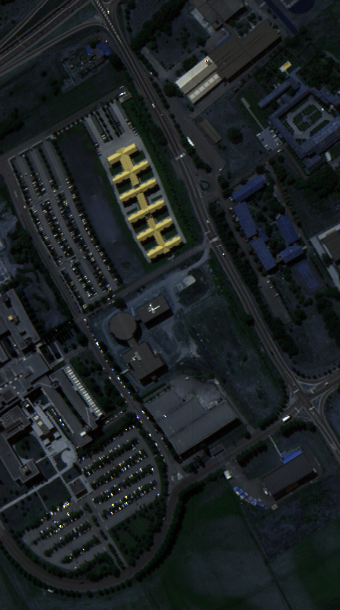} \\
        \includegraphics[width=0.09\linewidth]{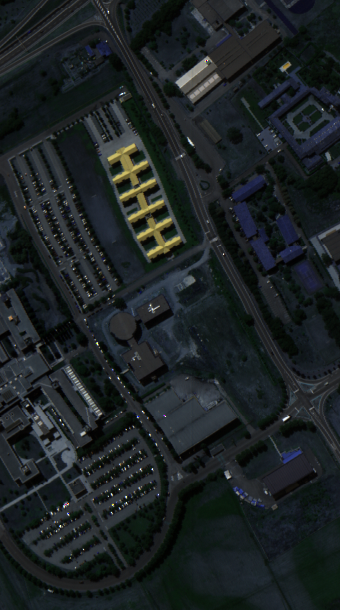} & 
		\includegraphics[width=0.09\linewidth]{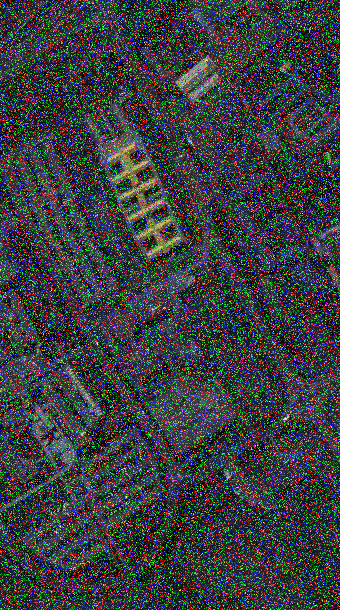} & \includegraphics[width=0.09\linewidth]{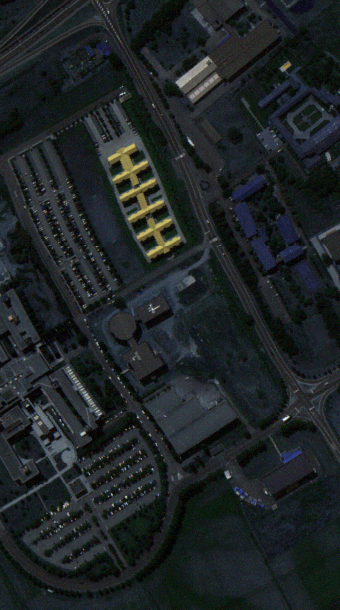} & \includegraphics[width=0.09\linewidth]{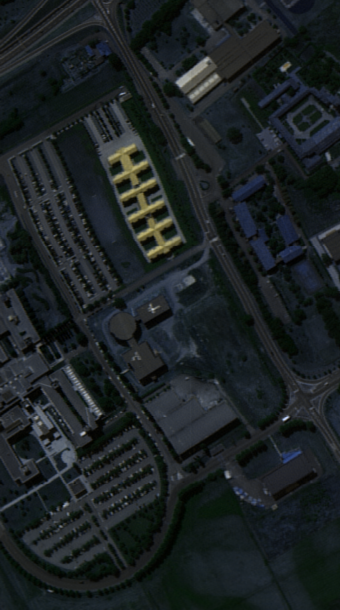}& \includegraphics[width=0.09\linewidth]{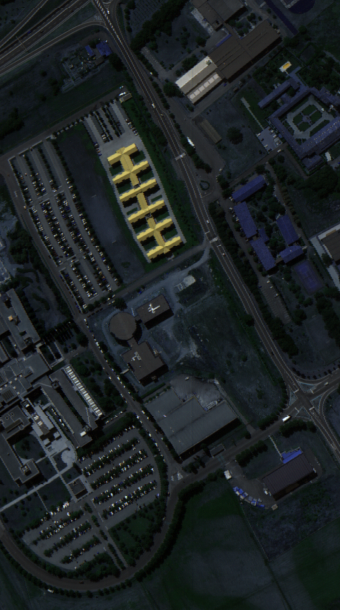}& \includegraphics[width=0.09\linewidth]{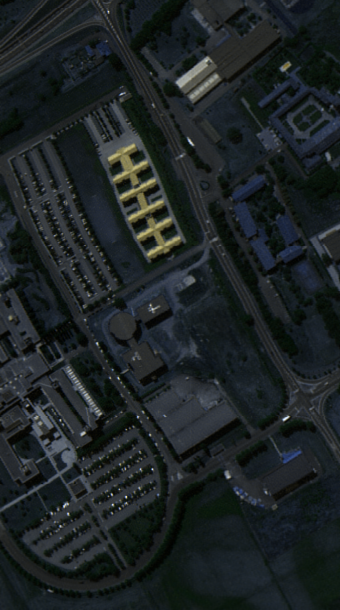}& \includegraphics[width=0.09\linewidth]{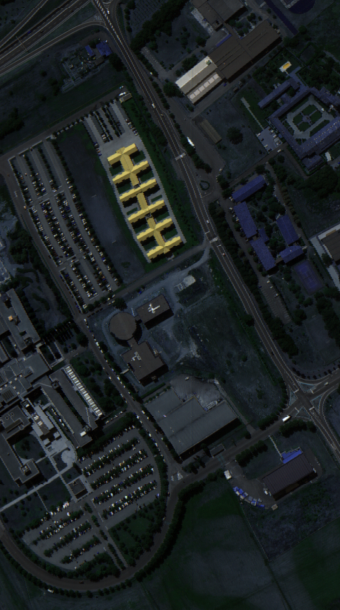}& \includegraphics[width=0.09\linewidth]{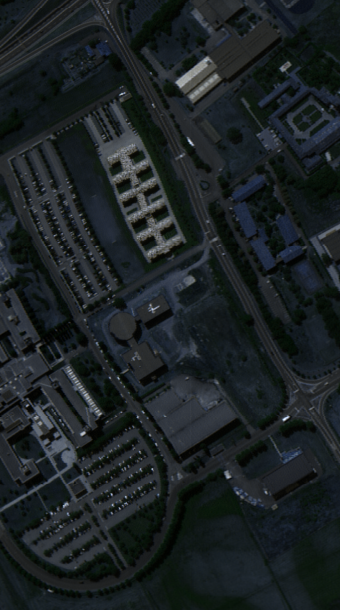}& \includegraphics[width=0.09\linewidth]{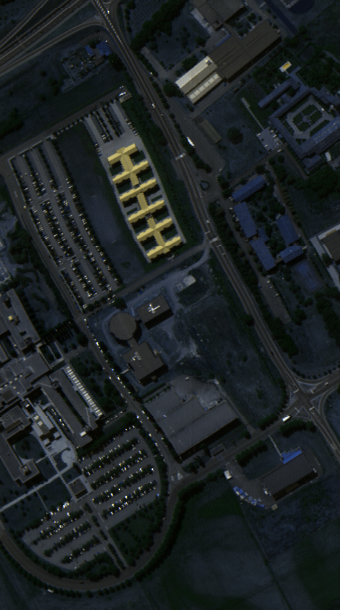}& \includegraphics[width=0.09\linewidth]{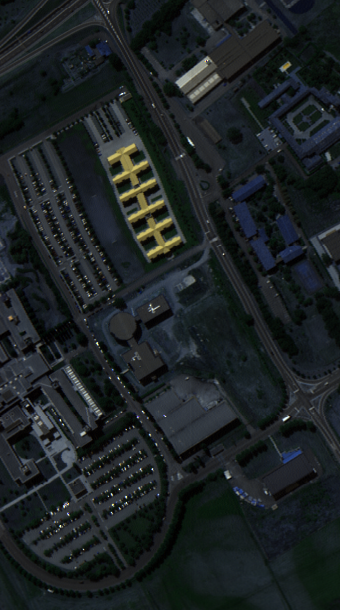} \\
	   \includegraphics[width=0.09\linewidth]{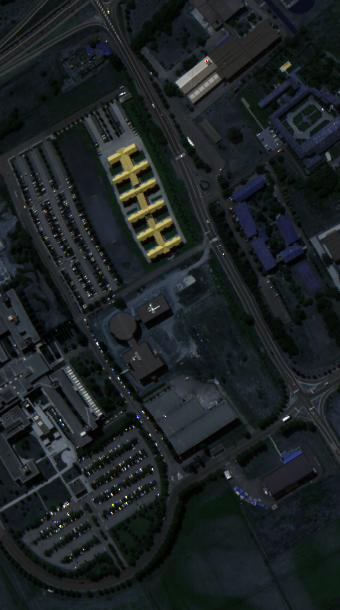}	&
		\includegraphics[width=0.09\linewidth]{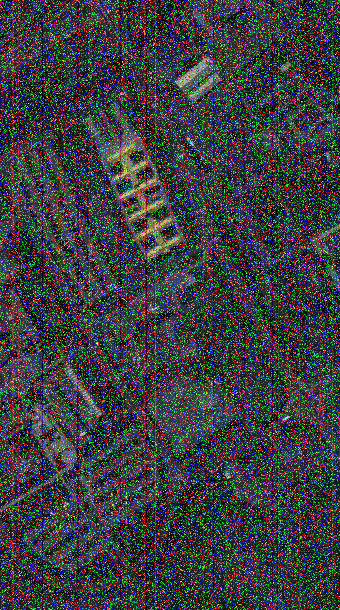} & \includegraphics[width=0.09\linewidth]{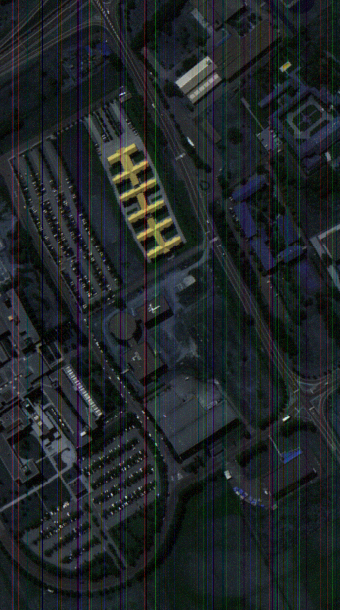} & \includegraphics[width=0.09\linewidth]{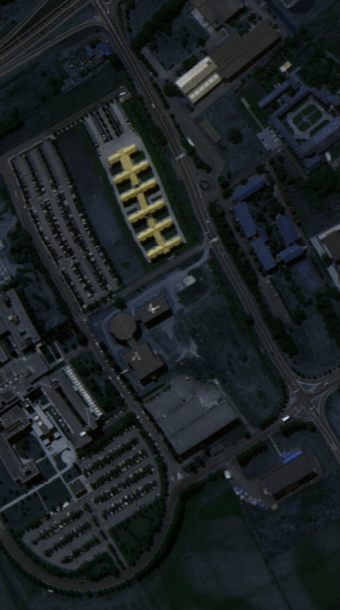}& \includegraphics[width=0.09\linewidth]{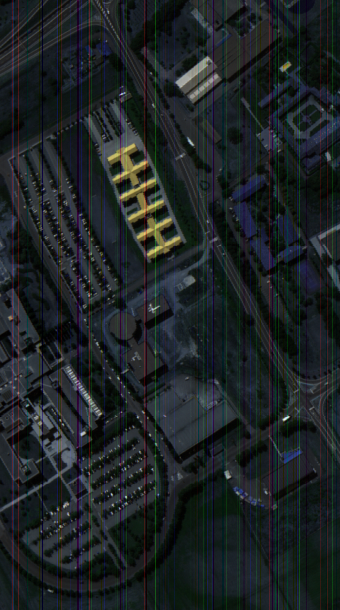}& \includegraphics[width=0.09\linewidth]{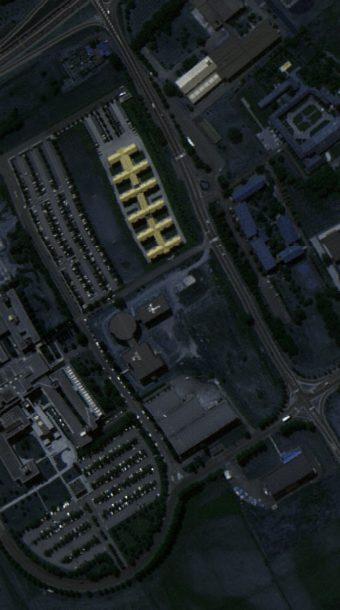}& \includegraphics[width=0.09\linewidth]{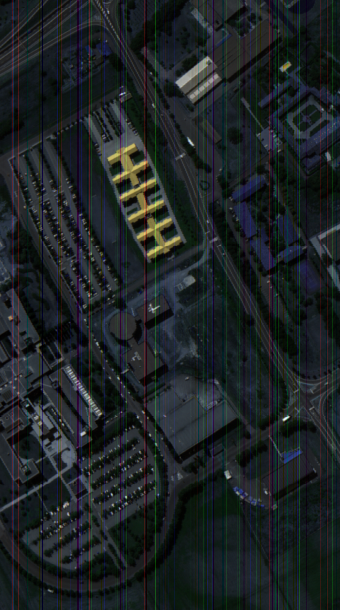}& \includegraphics[width=0.09\linewidth]{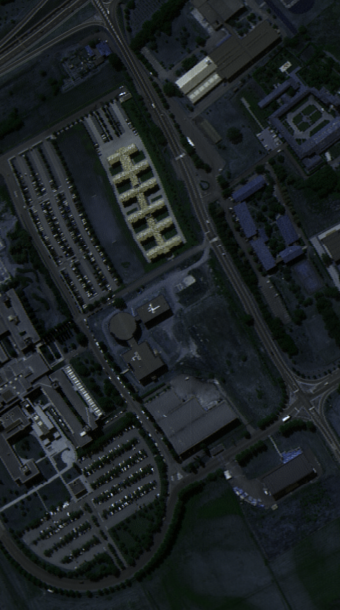}& \includegraphics[width=0.09\linewidth]{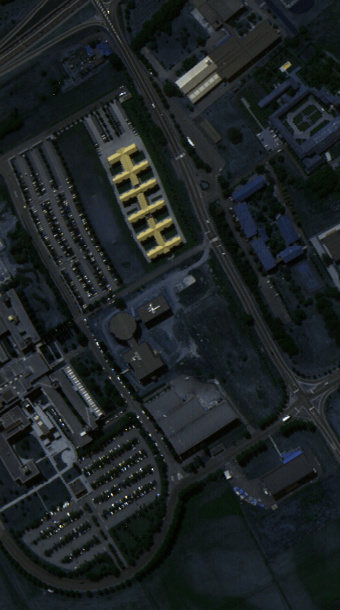}& \includegraphics[width=0.09\linewidth]{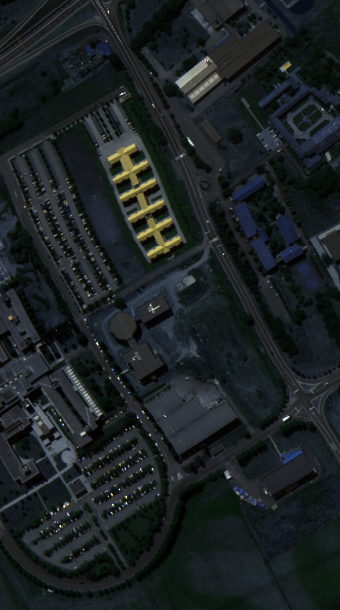} \\
		\footnotesize Original & \footnotesize Observed & \footnotesize TNN & \footnotesize KBR & \footnotesize ETRPCA & \footnotesize FTTNN & \footnotesize t-CTV & \footnotesize RTCUR-FF & \footnotesize Ours & \footnotesize Oracle
	\end{tabular}
	\caption{Visual results on the Indian Pines and PaviaU datasets with different noise types. Each group shows three rows: the first row corresponds to 30\% salt-and-pepper noise, the second to 30\% random noise, and the third to a combination of 30\% random and stripe noise. The 20th, 30th, and 50th bands are visualized.}
 \label{fig:HSI}
\end{figure*}

\begin{table*}[htbp]
	\centering\footnotesize\tabcolsep=1.mm
	\caption{Performance comparison of various methods for HSI denoising under different noise types.}
	\label{tab:HSI}
	\begin{tabular}{c|c|c|c|c|c|c|c|c|c|c||c}
		\hline
		HSIs & Noise & Metrics & Observed & TNN & KBR & ETRPCA & FTTNN & t-CTV & RTCUR-FF & Ours & Oracle \\ \hline
		\multirow{10}{*}{\makecell{Indian Pines\\$(145\times145\times220)$}} 
		& \multirow{3}{*}{\makecell{salt-and\\-pepper}} 
		& PSNR$\uparrow$ & 9.98 & 37.25 & 37.83 & 38.01 & 37.59 & \underline{38.88} & 37.81 & \textbf{41.03} & 41.03 \\ 
		& & SSIM$\uparrow$ & 0.0207 & 0.9574 & 0.9598 & 0.9613 & 0.9608 & \underline{0.9661} & 0.9612 & \textbf{0.9728} & 0.9727 \\ 
		& & ERGAS$\downarrow$ & 77.40 & 2.20 & 2.14 & 2.04 & 2.15 & \underline{1.85} & 2.11 & \textbf{1.49} & 1.49 \\ \cline{2-12}
		& \multirow{3}{*}{random} 
		& PSNR$\uparrow$ & 12.97 & 36.55 & 37.71 & 37.39 & 37.23 & \underline{38.37} & 37.97 & \textbf{39.19} & 41.04 \\ 
		& & SSIM$\uparrow$ & 0.0506 & 0.9516 & 0.9587 & 0.9559 & 0.9586 & \textbf{0.9639} & \underline{0.9609} & \textbf{0.9639} & 0.9730 \\ 
		& & ERGAS$\downarrow$ & 57.12 & 2.44 & 2.18 & 2.25 & 2.23 & \underline{1.97} & 2.08 & \textbf{1.84} & 1.48 \\ \cline{2-12}
		& \multirow{3}{*}{\makecell{random,\\stripe}} 
		& PSNR$\uparrow$ & 12.89 & 29.34 & \underline{37.70} & 29.29 & 36.47 & 29.52 & 36.30 & \textbf{38.91} & 40.90 \\ 
		& & SSIM$\uparrow$ & 0.0478 & 0.8586 & \underline{0.9586} & 0.8606 & 0.9569 & 0.8698 & 0.9455 & \textbf{0.9600} & 0.9722 \\ 
		& & ERGAS$\downarrow$ & 57.32 & 5.25 & \underline{2.18} & 5.29 & 2.39 & 5.11 & 2.52 & \textbf{1.89} & 1.50 \\ \cline{2-12}
		& \multicolumn{2}{c|}{Avg Time (s)$\downarrow$} 
		& - & 109.30 & 76.10 & 288.82 & \underline{29.96} & 210.65 & 1267.56 & \textbf{18.58} & 19.42 \\ \hline
		\multirow{10}{*}{\makecell{PaviaU\\$(610\times340\times103)$}} 
		& \multirow{3}{*}{\makecell{salt-and\\-pepper}} 
		& PSNR$\uparrow$ & 9.56 & 36.47 & 34.35 & 36.97 & 35.70 & \underline{38.09} & 35.31 & \textbf{40.35} & 40.67 \\ 
		& & SSIM$\uparrow$ & 0.0456 & 0.9526 & 0.9462 & 0.9716 & 0.9473 & \textbf{0.9772} & 0.9598 & \underline{0.9737} & 0.9734 \\ 
		& & ERGAS$\downarrow$ & 86.54 & 3.81 & 4.97 & 3.62 & 4.18 & \underline{3.25} & 4.61 & \textbf{2.52} & 2.44 \\ \cline{2-12}
		& \multirow{3}{*}{random} 
		& PSNR$\uparrow$ & 12.17 & 32.25 & 34.34 & 35.95 & 35.34 & \underline{37.32} & 34.47 & \textbf{37.87} & 40.68 \\ 
		& & SSIM$\uparrow$ & 0.0863 & 0.8569 & 0.9461 & 0.9645 & 0.9455 & \textbf{0.9740} & 0.9608 & \underline{0.9646} & 0.9735 \\ 
		& & ERGAS$\downarrow$ & 64.93 & 6.38 & 4.98 & 4.22 & 4.38 & \underline{3.62} & 5.15 & \textbf{3.35} & 2.43 \\ \cline{2-12}
		& \multirow{3}{*}{\makecell{random,\\stripe}} 
		& PSNR$\uparrow$ & 12.09 & 27.97 & 34.33 & 29.02 & \underline{35.33} & 29.27 & 35.36 & \textbf{37.76} & 40.65 \\ 
		& & SSIM$\uparrow$ & 0.0852 & 0.7630 & 0.9461 & 0.8521 & 0.9449 & 0.8595 & \underline{0.9625} & \textbf{0.9633} & 0.9733 \\ 
		& & ERGAS$\downarrow$ & 65.57 & 10.94 & 4.99 & 9.86 & \underline{4.40} & 9.62 & 4.63 & \textbf{3.39} & 2.45 \\ \cline{2-12}
		& \multicolumn{2}{c|}{Avg Time (s)$\downarrow$} 
		& - & 679.39 & 625.09 & 1887.70 & \underline{156.04} & 1250.36 & 23184.31 & \textbf{127.96} & 119.23 \\ \hline
	\end{tabular}
\end{table*}

\section{Conclusion}\label{section:conclusion}

This paper presents an outlier-aware TRPCA framework that addresses the limitations of traditional sparse outlier modeling through self-guided data augmentation and dynamic weight adaptation. By decoupling outlier suppression from low-rank factorization, our method achieves robust recovery of low-rank structures even in the presence of dense or spatially correlated outliers. Theoretical convergence guarantees ensure stability, while the absence of SVD in the optimization process enhances computational efficiency. Experimental results on synthetic and real-world datasets validate the framework’s superiority, particularly in scenarios with high outlier densities or complex noise patterns.

\bibliographystyle{IEEEtran}
\bibliography{IEEEabrv,ref}

\end{document}

\title{Outlier-aware Tensor Robust Principal Component Analysis with Self-guided Data Augmentation}

\author{Yangyang Xu, Kexin Li, Li Yang,  You-Wei Wen}

\markboth{Journal of \LaTeX\ Class Files,~Vol.~14, No.~8, August~2021}%
{Shell \MakeLowercase{\textit{et al.}}: A Sample Article Using IEEEtran.cls for IEEE Journals}


\maketitle

\section{Detailed Proof of Convergence Analysis}\label{section:convergence}

In this supplementary section, we provide a detailed proof of the convergence analysis. We begin by introducing key definitions and notations used throughout the proof, followed by the construction of surrogate functions and their properties, culminating in a convergence theorem based on the Block Successive Upper-bound Minimization (Block SUM) \cite{razaviyayn2013unified} framework.

\subsection{Key Notations and Definitions}\label{subsection:notations}

We first define the core functions and notations used in our convergence analysis. Since $\cL = (\bU_1, \bU_2, \bU_3)\cdot\cG$, we simplify $\|(\bU_1, \bU_2, \bU_3)\cdot\cG - \cY\|_{F}^2$ as $\|\cL - \cY\|_{F}^2$ and the sequence $\{\cY_t,\bU_{1t},\bU_{2t},\bU_{3t},\cG_t\}$  as the sequence $\{\cY_t,\cL_t\}$  when no ambiguity arises. 
We begin our analysis with the original objective function:
\begin{equation}\label{DAobjfunW}
 \Phi(\cY,\cL; \omega(\cZ)) = \lambda\|\cY-\cX\|_{\omega(\cZ)}^2 + \|\cL - \cY\|_F^2,
\end{equation}
where $\omega(\cZ_{ijk}) = \exp\bigl(-\tfrac{(\cZ_{ijk}-\cX_{ijk})^2}{2\gamma}\bigr)$ defines the weights.

Next, we define the Welsch function:
\begin{equation}
\psi(\cY) = 2\gamma\sum_{ijk}\bigl(1-\exp\bigl(-\tfrac{(\cY_{ijk}-\cX_{ijk})^2}{2\gamma}\bigr)\bigr).
\end{equation}
where $\gamma > 0$ is a tuning parameter controlling the function's sensitivity. This function has a quadratic surrogate:
\begin{equation}
\widehat{\psi}(\cY; \cZ) = \psi(\cZ) + \|\cY-\cX\|_{\omega(\cZ)}^2-\|\cZ-\cX\|_{\omega(\cZ)}^2.
\end{equation}

Based on the above, the complete objective function is formulated as:
\begin{align}
\Psi(\cY,\cL) &= \|\cL - \cY\|_{F}^2 + \lambda \psi(\cY). \label{objfunPsi}
\end{align}
To facilitate convergence analysis, we introduce the following surrogate functions:
\begin{align}
    \Phi_0(\cY,\cL; \cZ) &=\widehat{\Psi}(\cY,\cL; \cZ)= \|\cL - \cY\|_{F}^2 + \lambda \widehat{\psi}(\cY; \cZ), \label{surrogate Y}\\
    \Phi_{i}(\cY,\bU_1,\bU_2,\bU_3,\cG;\bC_{i})&=\Psi(\cY,\bU_1,\bU_2,\bU_3,\cG)+\alpha_1\left\|\bU_{i}-\bC_{i}\right\|_{F}^{2} ,\  i=1,2,3,\label{surrogate U}\\
    \Phi_{4}(\cY,\bU_1,\bU_2,\bU_3,\cG;\cD)&=\Psi(\cY,\bU_1,\bU_2,\bU_3,\cG)
    +\alpha_2^3\left\|\cG-\cD\right\|_{F}^{2},\label{surrogate G}
\end{align}
where $\omega(\cZ) = \exp\left(-\frac{(\cZ - \cX)^2}{2\gamma}\right)$, $\bU_i$ ($i=1,2,3$) are the factor matrices, $\cG$ is the core tensor, and $\alpha_1,\alpha_2 > 0$ are regularization parameters.

Finally, we recall two fundamental concepts that will be used in the convergence proof:

\begin{definition}[Quasi-convex function]
    A function $ g: \mathbb{R}^n \to \mathbb{R} $ is quasi-convex if for any $ \bx, \by \in \mathbb{R}^n $ and $ \theta \in [0, 1] $, the following holds:
    \begin{equation*}
        g(\theta \bx + (1 - \theta) \by) \leq \max\{g(\bx), g(\by)\}.
    \end{equation*}
\end{definition}

\begin{definition}[Directional derivative]
    Let $ g: \mathbb{R}^n \to \mathbb{R} $ be differentiable at $ \bx \in \mathbb{R}^n $. The directional derivative of $ g $ at $ \bx $ in the direction $ \bd \in \mathbb{R}^n $ is defined as
    $$g'(\bx; \bd) = \lim_{\theta \to 0} \frac{g(\bx + \theta \bd) - g(\bx)}{\theta}.$$
\end{definition}

% \noindent
% Welsch's function:
% \begin{equation}\label{Welschfun}
% \psi(a) = \gamma\left(1 - \exp\left(-\frac{a^2}{2\gamma}\right)\right).
% \end{equation}

% \noindent
% The objective function:
% \begin{equation}\label{DAobjfunW}
% \Phi(\cL,\cY; \cW) = \frac{\lambda}{2} \|\cY-\cX\|_{\cW}^2 + \frac{1}{2} \|\cL - \cY\|_F^2.
% \end{equation}

% \noindent
% The function to which the objective function converges:
% \begin{equation}\label{objfunPsi}
% \Psi(\cL,\cY) = \frac{1}{2}\|\cL - \cY\|_{F}^2 + \lambda \sum_{ijk} \psi\left(\cY_{ijk} - \cX_{ijk}\right). 
% \end{equation}

% \noindent
% Surrogate functions:
% \begin{align}
%     &\Phi_0(\cL,\cY;\cZ). \notag\\
%     &=\Phi(\cL,\cY; \omega(\cZ)) + \lambda\psi(\cZ-\cX)-\frac{\lambda}{2}\norm{\cB-\cX}_{\omega(\cZ)}^2 \label{surrogate Y}\\
%     &\Phi_{i}(\bU_1,\bU_2,\bU_3,\cG,\cY;\bC_{i}) \notag\\
%     &=\Psi(\bU_1,\bU_2,\bU_3,\cG,\cY)+\frac{\alpha}{2}\left\|\bU_{i}-\bC_{i}\right\|_{F}^{2} ,i=1,2,3\label{surrogate U}\\
%     &\Phi_{4}(\bU_1,\bU_2,\bU_3,\cG,\cY;\cD)\notag \\
%     &=\Psi(\bU_1,\bU_2,\bU_3,\cG,\cY)
%     +\frac{\alpha}{2}\left\|\cG-\cD\right\|_{2}^{2}. \label{surrogate G}
% \end{align}

\subsection{Detailed Proof of the Decreasing Property of $\Psi(\cY,\cL)$}\label{subsection:Surrogate Functions}

\subsubsection{Surrogate Function for Welsch's Function}

\begin{lemma}\label{lemma:inequality_Welsch}
Let $\psi(a)=2\gamma\bigl(1-\exp\bigl(-\tfrac{a^2}{2\gamma}\bigr)\bigr)$.  
The function $\varphi(a; b)$ is given by 
\begin{equation}\label{Welschsurrogate}
\varphi(a; b) = \psi(b) + \exp\left(-\frac{b^2}{2\gamma}\right)(a^2 - b^2).
\end{equation}
Then, $\varphi(a; b)$ satisfies the surrogate properties for $\psi(a)$, i.e.,
\begin{equation}\label{Welsch: inequality}
\psi(a) \leq \varphi(a; b) \quad \text{and} \quad \psi(a) = \varphi(a; a).
\end{equation}
\end{lemma}

\begin{proof}
    Consider the function $\psi_0(x) =2\gamma(1 - \exp(-\frac{x}{2\gamma}))$, then $\psi'_0(x)>0$ and $\psi''_0(x)<0$. Therefore, for any $a,b$, we have $\psi_0(a)\le \psi_0(b)+\exp(-\frac{b}{2\gamma})(a - b)$. By replacing $a$ and $b$ with $a^2$ and $b^2$ respectively, the inequality is proved.
\end{proof}

\begin{remark}
    By applying Lemma \ref{lemma:inequality_Welsch}, we deduce that
    \[
    \Psi(\cY,\cL)\leq \Phi_0(\cY,\cL;\cZ)
    \]
    holds for all $\cZ$, with equality achieved specifically when $\cZ = \cY$:
    \begin{equation}
    \Psi(\cY,\cL)=\Phi_{0}(\cY,\cL;\cY). 
    \end{equation}
    Thus, $\Phi_0(\cY, \cL; \cZ)$, as defined in \eqref{surrogate Y}, serves as a surrogate function of $\Psi(\cY, \cL)$.
\end{remark}
 
\begin{lemma}\label{lemma: Welsch's surrogate}
Let $\Phi(\cY,\cL; \cW)$ and $\Psi(\cY,\cL)$ be defined in \eqref{DAobjfunW} and \eqref{objfunPsi} respectively, and $\cW_{t}, \cY_{t+1}$ be the sequences generated by the proposed algorithm.
Then we have: 
\[
\cY_{t+1}=\argmin_{\cY\in\bB}\Phi(\cY,\cL_{t}; \omega(\cY_t)),
\]
where $\bB=\{\cY:\|\cY\|_{\infty}\le a\}$ for some constant $a > 0$. It then follows that
\begin{equation}\label{decreaseY}
	\Psi(\cY_{t+1},\cL_{t})\leq \Psi(\cY_{t},\cL_{t}).
\end{equation}
\end{lemma}

\begin{proof}
    % The proof follows directly from the use of the surrogate function for $\Psi(\cL,\cY)$ and the inequality $\Phi(\cL_{t},\cY_{t+1}; \cW_{t})\leq \Phi(\cL_{t},\cY_{t}; \cW_{t})$.
    Since $\Phi(\cY_{t+1}, \cL_{t}; \cW_{t}) \leq \Phi(\cY_{t}, \cL_{t}; \cW_{t})$, it follows that $\Phi_0(\cY_{t+1},\cL_t;\cY_{t})\le \Phi_0(\cY_t,\cL_{t};\cY_{t})$. In addition, since $\Phi_0(\cY,\cL;\cY)$ is a surrogate function of $\Psi(\cY,\cL)$, we have
    \begin{align*}
        \Psi(\cY_{t+1},\cL_t)\le \Phi_0(\cY_{t+1},\cL_t;\cY_t) \le \Phi_0(\cY_{t},\cL_t;\cY_t)=\Psi(\cY_t,\cL_t).
    \end{align*}
\end{proof}

\subsubsection{Subproblems for Factor Matrices}

% By sequentially using the relationships between the objective function $\Psi(\bU_1,\bU_2,\bU_3,\cG,\cY)$ and the surrogate functions $\Phi_{i}(\bU_1,\bU_2,\bU_3,\cG,\cY;\bC_{i})(i=1,2,3)$, we have the following lemma:

\begin{lemma}\label{lemma:surrogate U}
Let $\Psi(\cY,\cL)=\Psi(\cY,\bU_1,\bU_2,\bU_3,\cG)$ be defined in \eqref{objfunPsi}, and $\bU_{i,t+1}(i=1,2,3)$ be the sequences generated by the proposed algorithm. Then:
\begin{eqnarray*}
\bU_{1,t+1}&=&\argmin_{\bU_1}\Phi_1(\cY_{t+1},\bU_{1},\bU_{2,t},\bU_{3,t},\cG_t;\bU_{1,t}),\\
\bU_{2,t+1}&=&\argmin_{\bU_2}\Phi_2(\cY_{t+1},\bU_{1,t+1},\bU_{2},\bU_{3,t},\cG_t;\bU_{2,t}),\\
\bU_{3,t+1}&=&\argmin_{\bU_3}\Phi_3(\cY_{t+1},\bU_{1,t+1},\bU_{2,t+1},\bU_{3},\cG_t;\bU_{3,t})
\end{eqnarray*}
and
\begin{align*}
        \Psi(\cY_{t+1},\bU_{1,t+1},\bU_{2,t+1},\bU_{3,t+1},\cG_t)\leq \Psi(\cY_{t+1},\bU_{1,t},\bU_{2,t},\bU_{3,t},\cG_t).
\end{align*}
\end{lemma}

\begin{proof}
    % By directly using \eqref{eq:Phi_U} and the following inequalities:
    % \begin{align*}
    %     \Phi_{1}&(\bU_{1,t+1},\bU_{2,t},\bU_{3,t},\cG_t,\cY_{t+1};\bU_{1,t}) \\
    %     &\le \Phi_{1}(\bU_{1,t},\bU_{2,t},\bU_{3,t},\cG_t,\cY_{t+1};\bU_{1,t}),\\
    %     \Phi_{2}&(\bU_{1,t+1},\bU_{2,t+1},\bU_{3,t},\cG_t,\cY_{t+1};\bU_{2,t})\\
    %     &\le \Phi_{2}(\bU_{1,t+1},\bU_{2,t},\bU_{3,t},\cG_t,\cY_{t+1};\bU_{2,t}),\\
    %     \Phi_{3}&(\bU_{1,t+1},\bU_{2,t+1},\bU_{3,t+1},\cG_t,\cY_{t+1};\bU_{3,t})\\
    %     &\le \Phi_{3}(\bU_{1,t+1},\bU_{2,t+1},\bU_{3,t},\cG_t,\cY_{t+1};\bU_{3,t}),
    % \end{align*}
    % the lemma is proved.
    First, for $\bU_3$:
    \begin{align*}
    \Phi_3(&\cY_{t+1}, \bU_{1,t+1}, \bU_{2,t+1}, \bU_{3,t+1}, \cG_t; \bU_{3,t+1}) \\
    &\leq \Phi_3(\cY_{t+1}, \bU_{1,t+1}, \bU_{2,t+1}, \bU_{3,t+1}, \cG_t; \bU_{3,t}) \\
    &\leq \Phi_3(\cY_{t+1}, \bU_{1,t+1}, \bU_{2,t+1}, \bU_{3,t}, \cG_t; \bU_{3,t}),
    \end{align*}
    where the first inequality uses the definition of the surrogate function, and the second follows from the minimization property of $\bU_{3,t+1}=\argmin_{\bU_3}\Phi_3(\cY_{t+1},\bU_{1,t+1},\bU_{2,t+1},\bU_{3},\cG_t;\bU_{3,t})$.
    Similarly, for $\bU_2$:
    \begin{align*}
    \Phi_2(&\cY_{t+1}, \bU_{1,t+1}, \bU_{2,t+1}, \bU_{3,t}, \cG_t; \bU_{2,t+1}) \\
    &\leq \Phi_2(\cY_{t+1}, \bU_{1,t+1}, \bU_{2,t+1}, \cG_t, \bU_{3,t}; \bU_{2,t}) \\
    &\leq \Phi_2(\cY_{t+1}, \bU_{1,t+1}, \bU_{2,t}, \bU_{3,t}, \cG_t; \bU_{2,t}),
    \end{align*}
    and for $\bU_1$:
    \begin{align*}
    \Phi_1(&\cY_{t+1}, \bU_{1,t+1}, \bU_{2,t}, \bU_{3,t}, \cG_t; \bU_{1,t+1}) \\
    &\leq \Phi_1(\cY_{t+1}, \bU_{1,t+1}, \bU_{2,t}, \bU_{3,t}, \cG_t; \bU_{1,t}) \\
    &\leq \Phi_1(\cY_{t+1}, \bU_{1,t}, \bU_{2,t}, \bU_{3,t}, \cG_t; \bU_{1,t}).
    \end{align*}
    % \begin{align*}
    %     \Phi_3&(\cY_{t+1},\bU_{1,t+1},\bU_{2,t+1},\bU_{3,t+1},\cG_t;\bU_{3,t+1}) \\
    %     &\le \Phi_3(\cY_{t+1},\bU_{1,t+1},\bU_{2,t+1},\bU_{3,t+1},\cG_t;\bU_{3,t}) \\
    %     &\le \Phi_3(\cY_{t+1},\bU_{1,t+1},\bU_{2,t+1},\bU_{3,t},\cG_t;\bU_{3,t}),\\
    %     \Phi_2&(\cY_{t+1},\bU_{1,t+1},\bU_{2,t+1},\bU_{3,t},\cG_t;\bU_{2,t+1}) \\
    %     &\le \Phi_2(\cY_{t+1},\bU_{1,t+1},\bU_{2,t+1},\bU_{3,t},\cG_t;\bU_{2,t}) \\
    %     &\le \Phi_2(\cY_{t+1},\bU_{1,t+1},\bU_{2,t},\bU_{3,t},\cG_t;\bU_{2,t}),\\
    %     \Phi_1&(\cY_{t+1},\bU_{1,t+1},\bU_{2,t},\bU_{3,t},\cG_t;\bU_{1,t+1}) \\
    %     &\le \Phi_1(\cY_{t+1},\bU_{1,t+1},\bU_{2,t},\bU_{3,t},\cG_t;\bU_{1,t}) \\
    %     &\le \Phi_1(\cY_{t+1},\bU_{1,t},\bU_{2,t},\bU_{3,t},\cG_t;\bU_{1,t}).\\
    % \end{align*}
    Next, we use the consistency of the surrogate functions across updates:
    \begin{align*}
        \Phi_3(\cY_{t+1},\bU_{1,t+1},\bU_{2,t+1},\bU_{3,t},\cG_t;\bU_{3,t}) &=\Phi_2(\cY_{t+1},\bU_{1,t+1},\bU_{2,t+1},\bU_{3,t},\cG_t;\bU_{2,t+1}),\\
        \Phi_2(\cY_{t+1},\bU_{1,t+1},\bU_{2,t},\bU_{3,t},\cG_t;\bU_{2,t})&=\Phi_1(\cY_{t+1},\bU_{1,t+1},\bU_{2,t},\bU_{3,t},\cG_t;\bU_{1,t+1}),
    \end{align*}
    which follow from the definitions of $\Phi_i$ aligning with $\Psi$ at the current iterates. Combining these, we obtain:
    \begin{align*}
        \Psi&(\cY_{t+1},\bU_{1,t+1},\bU_{2,t+1},\bU_{3,t+1},\cG_t)\\
        &=\Phi_3(\cY_{t+1},\bU_{1,t+1},\bU_{2,t+1},\bU_{3,t+1},\cG_t;\bU_{3,t+1}) \\
        &\le \Phi_1(\cY_{t+1},\bU_{1,t},\bU_{2,t},\bU_{3,t},\cG_t;\bU_{1,t})
        \\
        &=\Psi(\cY_{t+1},\bU_{1,t},\bU_{2,t},\bU_{3,t},\cG_t).
    \end{align*}
    Thus, the inequality holds.
\end{proof}

\subsubsection{Subproblem for Core Tensor}

\begin{lemma}\label{lemma:surrogate G}
Let $\Psi(\cY,\cL)=\Psi(\cY,\bU_1,\bU_2,\bU_3,\cG)$ be defined in \eqref{objfunPsi}, and $\cG_{t+1}$ be the sequences generated by the proposed algorithm. 
Then we have:
\[
\cG_{t+1}=\argmin_{\cG}\Phi_4(\cY_{t+1},\bU_{1,t+1},\bU_{2,t+1},\bU_{3,t+1},\cG;\cG_t)
\]
and
\begin{align*}
        \Psi(\cY_{t+1},\bU_{1,t+1},\bU_{2,t+1},\bU_{3,t+1},\cG_{t+1})\leq \Psi(\cY_{t+1},\bU_{1,t+1},\bU_{2,t+1},\bU_{3,t+1},\cG_t).
    \end{align*}
\end{lemma}

\begin{proof}
    The proof mirrors the approach in Lemma \ref{lemma:surrogate U}, relying on the surrogate function $\Phi_4$ to establish the decrease of $\Psi$. Specifically, consider the following chain of relationships:
    \begin{align*}
        \Psi&(\cY_{t+1},\bU_{1,t+1},\bU_{2,t+1},\bU_{3,t+1},\cG_{t+1})\\
        &=\Phi_4(\cY_{t+1},\bU_{1,t+1},\bU_{2,t+1},\bU_{3,t+1},\cG_{t+1};\cG_{t+1})\\
        &\le \Phi_4(\cY_{t+1},\bU_{1,t+1},\bU_{2,t+1},\bU_{3,t+1},\cG_{t+1};\cG_{t})\\
        &\le \Phi_4(\cY_{t+1},\bU_{1,t+1},\bU_{2,t+1},\bU_{3,t+1},\cG_{t};\cG_{t})\\
        &= \Psi(\cY_{t+1},\bU_{1,t+1},\bU_{2,t+1},\bU_{3,t+1},\cG_t).
    \end{align*}
\end{proof}

By combining Lemmas \ref{lemma: Welsch's surrogate}, \ref{lemma:surrogate U} and \ref{lemma:surrogate G}, it can be concluded that the objective function $\Psi(\cY,\cL)$ is decreasing:

\begin{lemma}
    Let $\Psi(\cY,\cL)$ be defined in \eqref{objfunPsi}, and $\cY_{t+1},\cL_{t+1}$ be the sequences generated by the proposed algorithm. Then we have
    \begin{align*}
        \Psi(\cY_{t+1},\cL_{t+1})\leq \Psi(\cY_{t},\cL_{t}).
    \end{align*}
\end{lemma}

\subsection{Block Successive Upper-bound Minimization}

To establish the global convergence of the proposed algorithm, we leverage the Block SUM method, which provides a framework for analyzing iterative optimization with interdependent variables. Below, we adapt this method to our problem by aligning the updates of $\cY$, $\bU_i$, and $\cG$ with its block-wise structure. Before proving the convergence, we first introduce the Block SUM method.
% Before proving the convergence, we first introduce the Block Successive Upper-bound Minimization (Block SUM) method \cite{razaviyayn2013unified}.

Suppose that the feasible set $\bOmega$ is the cartesian product of $n$ closed convex sets: $\bOmega=\bOmega_{1}\times \cdots\times\bOmega_{n}$, with $\bOmega_{i}\subset \mathbb{R}^{m_i}$ and $\sum_{i}m_{i}=m$.  Accordingly, the optimization variable $\bx\in\mathbb{R}^{m}$ can be decomposed as $\bx=(\bx_1,\bx_2,\ldots,\bx_n)$, with $\bx_i\in \bOmega_i,i=1,\ldots,n$. In our problem, we treat $\cY$, $\bU_1$, $\bU_2$, $\bU_3$, and $\cG$ as separate blocks, corresponding to $\bx_i$ in the Block SUM framework. Block SUM considers the following problem:
\begin{equation}\label{BSUM problem}
\begin{aligned}
\min_{\bx_1,\bx_2,\ldots,\bx_n} &\quad \Psi(\bx_1,\bx_2,\ldots,\bx_n) \\ 
\text{s.t.} &\quad \bx_{i}\in \bOmega_i,\quad i=1,2,\ldots,n.
\end{aligned}
\end{equation}
When $\Psi(\bx)$ is non-smooth, non-convex, or exhibits strong coupling between blocks, directly solving problem \eqref{BSUM problem} becomes challenging. To address this, BSUM sequentially optimizes a series of surrogate functions that approximate $g$ while possessing more favorable properties. Specifically, at iteration $t+1$, BSUM updates the blocks cyclically (i.e., $i = 1, 2, \ldots, n$), where the $i$-th block $\bx_i$ is obtained by solving the subproblem:
\begin{equation}\label{BSUM subproblem}
\begin{aligned}
\min_{\bx_i} &\quad \Phi_{i}(\bx_1^{t+1},\ldots,\bx_{i-1}^{t+1},\bx_i,\bx_{i+1}^{t},\ldots,\bx_n^t;\by_{i}) \\ 
\text{s.t.} &\quad \bx_i\in \bOmega_{i},
\end{aligned}
\end{equation}
with $\by_i$ being an auxiliary variable for the $i$-th block, typically set to $\by_i = \bx_i^t$, the value of the $i$-th block from the previous iteration. We refer to $\Phi_i(\bx;\by_i)$ as a surrogate function for $\Psi(\bx)$, which must satisfy the following assumptions to ensure convergence:

\begin{assumption}[Assumption 2 of \cite{razaviyayn2013unified}]\label{Block MM assumption}

Suppose that for all $i=1,\ldots,n$, $\bx_i,\by_i\in \bOmega_i$, and $\bx=(\bx_1,\bx_2,\ldots,\bx_n)\in\mathbb{R}^{m}$. Then each surrogate function $\Phi_{i}(\bx;\by_i)$ satisfies the following assumptions:
\begin{subequations} 
\begin{align}
& \Phi_i\left(\bx;\bx_i\right)=\Psi(\bx),\label{assumption1} \\
& \Phi_i\left(\bx;\by_i\right) \geq \Psi\left(\bx\right),  \label{assumption2} \\
& \Phi_i'(\bx; \by_i; \bd_i)\big|_{\by_i=\bx_i} = \Psi'(\bx;\bd),\label{assumption3} \\
& \Phi_i\left(\bx;\by_i\right) \text{ is continuous in }(\bx,\by_i).\label{assumption4}
\end{align}
\end{subequations}
Here, $\bd \in \mathbb{R}^m$ is a direction vector with $\bd_i \in \mathbb{R}^{m_i}$ as its $i$-th block and all other blocks zero, i.e., $\bd=(0,\ldots,\bd_i,\ldots,0)$. $\Phi_i'(\bx; \by_i; \bd)$ is the directional derivative of $\Phi_i$ with respect to its first argument $\bx$ along direction $\bd_i$ (i.e., the directional derivative with respect to $\bx_i$), and $\Psi'(\bx; \bd)$ is the directional derivative of $\Psi$ with respect to $\bx$ along direction $\bd$.
\end{assumption}

Based on the definitions of the surrogate functions \eqref{surrogate Y}, \eqref{surrogate U} and \eqref{surrogate G}, we can draw the following conclusion:

\begin{lemma}\label{lemma:convergence1}
    The surrogate functions \eqref{surrogate Y}, \eqref{surrogate U} and \eqref{surrogate G} satisfy assumptions \eqref{assumption1}-\eqref{assumption4}.
\end{lemma}

\begin{proof}
    Assumptions \eqref{assumption1} and \eqref{assumption2} can be easily verified. For assumption \eqref{assumption3}, since $\Phi_0(\cY, \cL; \cZ)$ and $\Psi(\cY, \cL)$ are both differentiable with respect to $\cY$, it suffices to show that their gradients with respect to $\cY$ are equal at $\cZ = \cY$. We compute:
    $$\begin{aligned}
        \nabla_{\cY}\Phi_0(\cY,\cL;\cZ)\big|_{\cZ=\cY}
        &=2(\cY-\cL)+\lambda\nabla_{\cY}\left(\|\cY-\cX\|_{\omega(\cZ)}^2\right)\big|_{\cZ=\cY} \\
        &=\nabla_{\cY}\Psi(\cY,\cL).
    \end{aligned}$$
    Therefore, $\Phi_0(\cY,\cL;\cZ)$ satisfies assumption \eqref{assumption3}.
    Moreover, it is evident that $\Phi_0(\cY,\cL;\cZ)$ is also continuous, thus it satisfies assumptions \eqref{assumption1}-\eqref{assumption4}. Similarly, it can be verified that each $\Phi_i$ satisfies assumptions \eqref{assumption1}-\eqref{assumption4}, for $i=1,\ldots,4$.
\end{proof}

\begin{lemma}[Theorem 2 of \cite{razaviyayn2013unified}]\label{theorem: convergence}
Suppose $\bOmega$ is convex. Under assumption \ref{Block MM assumption} (for simplicity assume that $\Psi$ is continuously differentiable):
\begin{itemize}
\item[(a)] If $\Phi_i\left(\bx; \by_{i}\right)$ is quasi-convex in $\bx_i$ and each subproblem \eqref{BSUM subproblem} has a unique solution, then every limit point of the sequence $\{\bx^{t+1}\}$ generated by the $t+1$-th iteration of the block SUM is a coordinatewise minimum of \eqref{BSUM problem}.
\item[(b)] If the level set $\bOmega^0=\left\{\bx \mid \Psi(\bx) \leq \Psi\left(\bx^0\right)\right\}$ is compact, each subproblem \eqref{BSUM subproblem} has a unique solution for at least $n-1$ blocks, then 
$$\lim _{t \rightarrow \infty} d\left(\bx^{t+1}, \bOmega^{*}\right)=\lim _{t \rightarrow \infty}\inf_{\bx^* \in \bOmega^{*}}\|\bx^{t+1}-\bx^*\|=0,$$
where $\bOmega^{*}$ is the set of stationary points.
\end{itemize}
\end{lemma}

\subsection{Convergence}\label{subsection:convergence}

Since the surrogate function \eqref{surrogate Y} is a quadratic function with respect to $\cY$, it is strongly convex. Similarly, surrogate functions \eqref{surrogate U} and \eqref{surrogate G} are also strongly convex. Therefore, we have the following conclusion:

\begin{lemma}\label{lemma:convergence2}
% Each subproblem in the proposed model \eqref{optimization model} has a unique solution.
Each subproblem in the proposed model has a unique solution.
\end{lemma}

\begin{lemma}\label{lemma:convergence3}
The level set $\bOmega^0=\{(\cY,\cL)\mid \Psi(\cY,\cL)\le \Psi(\cY_0,\cL_0)\}$ is compact, where $\Psi(\cY,\cL)$ is defined in \eqref{objfunPsi}.
\end{lemma}

\begin{proof}
The second term of $\Psi(\cY, \cL)$ is clearly bounded, which in turn implies that the first term is also bounded. Combined with the fact that $\cY \in \bB$, it follows that the pair $(\cY, \cL)$ is bounded. Thus, the set $\bOmega$ is bounded as well. 
Since $\Psi(\cY, \cL)$ is continuous, the sublevel set $\bOmega^0$ is closed. Therefore, $\bOmega^0$ is compact.
\end{proof}

\begin{theorem}
Assuming the set $\bOmega$ is convex, let $\bx_{t+1}:=(\cY_{t+1},\cL_{t+1})\in \bOmega$ denote the variables in the $(t+1)$-th iteration of the proposed algorithm. Then every limit point of $\{\bx_{t+1}\}$ is a stationary point of $\Psi(\cY,\cL)$, i.e.,
\begin{equation}
\lim _{t \rightarrow \infty} d\left(\bx_{t+1}, \bOmega^{*}\right)=\lim _{t \rightarrow \infty}\inf_{\bx^* \in \bOmega^{*}}\|\bx_{t+1}-\bx^*\|=0,
\end{equation}
where $\bOmega^{*}$ is the set of stationary points of $\Psi(\cY,\cL)$.
\end{theorem}

\begin{proof}
Using Lemma \ref{lemma:convergence1}, Lemma \ref{theorem: convergence}, Lemma \ref{lemma:convergence2}, and Lemma \ref{lemma:convergence3}, the proof follows.
\end{proof}

\ifCLASSOPTIONcaptionsoff
  \newpage
\fi



%\bibliographystyle{IEEEtran}
\bibliographystyle{IEEEtran}
\bibliography{IEEEabrv,ref}


% \end{IEEEbiography}


% % insert where needed to balance the two columns on the last page with
% %\newpage



%\vfill

